\documentclass[a4paper]{amsart}

\usepackage[backrefs,lite]{amsrefs}
\usepackage{url}

\makeatletter
\DeclareFontEncoding{LS1}{}{}
\DeclareFontSubstitution{LS1}{stix}{m}{n}
\DeclareMathAlphabet{\skr}{LS1}{stixscr}{m}{n}
\makeatother

\usepackage{amssymb}              
\usepackage{wasysym}              
\usepackage{lmodern,bm}              
\usepackage{stmaryrd}
\usepackage[all]{xy}
\usepackage{tikz}
\usepackage{graphicx}
\usetikzlibrary{calc} 

\usepackage{amsmath,amssymb,amsthm}
\usepackage{tikz, pgfplots, pgfplotstable}
\usetikzlibrary{calc}
\usetikzlibrary{decorations.markings}
\usetikzlibrary{positioning,arrows}

\CompileMatrices

\newtheorem{theorem}{Theorem}[section]

\newtheorem{proposition}[theorem]{Proposition}
\newtheorem{corollary}[theorem]{Corollary}

\theoremstyle{definition}

\newtheorem{definition}[theorem]{Definition}
\newtheorem{construction}[theorem]{Construction}

\numberwithin{equation}{theorem}

\def\Cl{{\rm Cl}}
\def\cl{{\rm cl}}

\def\CC{{\mathbb C}}
\def\KK{{\mathbb K}}
\def\TT{{\mathbb T}}
\def\ZZ{{\mathbb Z}}

\def\QQ{{\mathbb Q}}
\def\PP{{\mathbb P}}

\def\conv{{\rm conv}}

\def\im{{\rm im}}

\def\bangle#1{{\langle #1 \rangle}}

\def\Pic{{\rm Pic}}

\def\Spec{{\rm Spec}}

\def\cone{{\rm cone}}

\def\lcm{{\rm lcm}}
\def\im{{\rm im}}

\DeclareMathOperator{\Mov}{\mathrm{Mov}}

\title[Full intrinsic quadrics of dimension two]{Full intrinsic quadrics of dimension two}

\author[J\"urgen Hausen, Katharina Kir\'{a}ly]{J\"urgen Hausen, Katharina Kir\'{a}ly}

\address{Mathematisches Institut, Universit\"at T\"ubingen,
Auf der Morgenstelle 10, 72076 T\"ubingen, Germany}
\email{juergen.hausen@uni-tuebingen.de}

\address{Mathematisches Institut, Universit\"at T\"ubingen,
Auf der Morgenstelle 10, 72076 T\"ubingen, Germany}
\email{kaki@math.uni-tuebingen.de}


\begin{document}

\begin{abstract}
A full intrinsic quadric is a normal complete
variety with a finitely generated Cox ring
defined by a single quadratic relation of
full rank.
We describe all surfaces of this type
explicitly via local Gorenstein indices.
As applications, we present upper and
lower bounds in terms of the Gorenstein
index for the degree, the log canonicity
and the Picard index.
Moreover, we determine all full intrinsic
quadric surfaces admitting a K\"ahler-Einstein
metric.
\end{abstract}

\maketitle

\section{Introduction}

The purpose of this article is to structure 
and explore the (infinite) playground of 
full intrinsic quadric (algebraic)
surfaces $X$, defined over an algebraically
closed field~$\KK$ of characteristic zero.
The name \emph{full intrinsic quadric}
refers to the property that the Cox ring
of $X$ is defined by a single homogeneous
quadratic relation of full rank; see~\cite{BeHa}.
Intrinsic quadrics exist as well in higher
dimensions and form an explicit example class
closely beneath the toric varieties which
are characterized by having a polynomial
ring as Cox ring; see~\cites{Bou, FaHa, Hi}
for sample work.

As we will see in Theorem~\ref{thm:fiqs-cr},
every full intrinsic quadric surface $X$ is
projective, normal, $\QQ$-factorial, rational
and allows a non-trivial action of the
multiplicative action of $\KK^*$.
This allows us to realize $X$ as a
surface in a specific toric threefold~$Z$.
More precisely, consider integral $3 \times n$
matrices $P$ of the form

{\small
\[
\left[
\begin{array}{cccc}
-1 & -1 & 2 & 0
\\
-1 & -1 & 0 & 2
\\
a  & b  & 1 & 1 
 \end{array}
\right],
\quad
\left[
\begin{array}{ccccc}
-1 & -1 & 1 & 1 & 0
\\
-1 & -1 & 0 & 0 & 2
\\
a  & b  & 0 & c & 1
 \end{array}
\right],
\quad
\left[
\begin{array}{cccccc}
-1 & -1 & 1 & 1 & 0 & 0
\\
-1 & -1 & 0 & 0 & 1 & 1
\\
a  & b  & 0 & c & 0 & d
 \end{array}
\right].
\]
}

\noindent
Given such $P$, fix a complete fan
$\Sigma$ in $\ZZ^3$ having the columns
of $P$ as its primitive ray generators,
let $Z$ be the associated toric threefold
and set 
\[
X
\ := \
\overline{ \{t \in \TT^3; \ 1+z_1+z_2 = 0\} }
\ \subseteq \
Z,
\]
where $\TT^3 \subseteq Z$ is the standard 3-torus
with the coordinates $z_1,z_2,z_3$.
Then $X$ is a full intrinsic quadric surface.
The Picard number of $X$ is $\rho(X) = n-3$
and~$X$ comes with the effective $\KK^*$-action
given on $X \cap \TT^3$ by
\[
t \cdot z \ = \ (z_1,z_2,tz_3).
\]

Our main results,
Theorems~\ref{thm:fiqs2matrix-rho1},~\ref{thm:fiqs2matrix-rho2}
and~\ref{thm:fiqs2matrix-rho3},
provide an \emph{explicit and redundance free
presentation of all full intrinsic quadric surfaces}
via their defining matrices $P$ in terms of local
Gorenstein indices and local class group orders
of the possibly singular points:
for each of the possible Picard numbers
$\rho(X)=1,2,3$, we find four infinite series,
each depending on two local Gorenstein indices,
$\iota^+$, $\iota^-$
and on $\rho(X)-1$ local class group orders,
bounded by $\iota^+$, $\iota^-$.

\goodbreak

All full intrinsic quadric surfaces $X$ turn out
to be log del Pezzo surfaces;
see Proposition~\ref{prop:fiqs2fano}.
We use our main results to study their geometry.
For instance,
Corollaries~\ref{cor:anticandeg},~\ref{cor:maxlc}
and~\ref{cor:picind} deliver upper and
lower bounds on the degree~$\mathcal{K}_X^2$,
the log canonicity $\varepsilon_X$
and the Picard index $\mathfrak{p}_X$, all in terms
of the Gorenstein index~$\iota_X$; in particular,
we obtain 
\[
\frac{2}{\iota_X} \le \mathcal{K}_X^2 \le \frac{9}{2} + \frac{9}{2\iota_X},
\quad
\frac{2}{\iota_X} \le \varepsilon_X \le \frac{3}{\sqrt{\iota_X}},
\quad  
\iota_X \le \mathfrak{p}_X \le 2\iota_X^2(4\iota_X-1)^2(2\iota_X-1).
\]
Another outcome of Theorems~\ref{thm:fiqs2matrix-rho1},
~\ref{thm:fiqs2matrix-rho2} and~\ref{thm:fiqs2matrix-rho3}
is the following explicit (infinite) list of all
full intrinsic quadric complex surfaces admitting
a K\"ahler-Einstein metric;
see Corollary~\ref{cor:ke-metrics} for the precise
formulation and more background.

\begin{corollary}
The full intrinsic quadric complex surfaces admitting
a K\"ahler-Einstein metric are precisely those constructed
from a matrix $P$ of the shape

{\small
\[
\begin{array}{lll}
\begin{array}{l}
\rho = 1, \ 2 \nmid \iota:
\\
\  
\end{array}
&\hspace*{.2cm}&
\begin{array}{l}
\rho = 3,
\quad
2 \nmid \iota,
\quad
-\iota+1 \, \le \, c \, \le \, -2, 
\\
\max(c, -2\iota - 2c)  \le  d \le  -\iota -1-c:
\end{array}
% \begin{array}{l}
% \rho = 3,
% \quad
% 2 \nmid \iota,
% \quad
% -2\iota \, \le \, 2c +d,
% \\
% c \, \le \, d \, \le \, -1,
% \
% c+d \, \le \, -\iota-1:
% \end{array}
%
\\[10pt]
\ %P =   
\left[
\begin{array}{cccc}
-1 & -1 & 2 & 0 
\\
-1 & -1 & 0 & 2
\\
\iota-1 & -\iota-1 & 1 & 1 
\end{array}
\right],
&\hspace*{1.2cm}&
\ %P =   
\left[
\begin{array}{cccccc}
-1 & -1 & 1 & 1 & 0 & 0
\\
-1 & -1 & 0 & 0 & 1 & 1
\\
\iota & -\iota -c-d & 0 & c & 0 & d 
\end{array}
\right],
\\[30pt]
\begin{array}{l}
\rho = 1, \ 4 \mid \iota:
\\
\
\end{array}
&\hspace*{1.2cm}&
\begin{array}{l}
\rho=3,
\quad
-2\iota + 1 \, \le \, c \, \le \, -2,
\\
\max(c, -4\iota-2c) \le d \le -2\iota -1-c:
\end{array}
\\[10pt]
\ %P =   
\left[
\begin{array}{cccc}
-1 & -1 & 2 & 0 
\\
-1 & -1 & 0 & 2
\\
\frac{\iota}{2}-1 & -\frac{\iota}{2}-1 & 1 & 1 
\end{array}
\right],
&\hspace*{1.2cm}&
\ %P =   
\left[
\begin{array}{cccccc}
-1 & -1 & 1 & 1 & 0 & 0
\\
-1 & -1 & 0 & 0 & 1 & 1
\\
2\iota & -2\iota -c -d & 0 & c & 0 & d 
\end{array}
\right].
\end{array}    
\]
}

\medskip
\noindent
Here, $\rho = 1,3$ is the Picard number
and $\iota \in \ZZ_{\ge 1}$ the Gorenstein index of the
resulting full intrinsic quadric complex surface
$X$ arising from the matrix $P$.
\end{corollary}

Finally, our description yields a filtration of the
whole infinite class of full
intrinsic quadric surfaces into finite subclasses
by bounding the Gorenstein index.
This allows, for instance, counting results
as the following.

\goodbreak

\begin{corollary}
Up to isomorphy, there are precisely
15$\,$538$\,$339 full intrinsic
quadric surfaces of Gorenstein index
at most 200.

\begin{enumerate}
\item
In Picard number one, we find in total 883 full intrinsic
quadric surfaces of Gorenstein index at most 200, filtered
as follows:
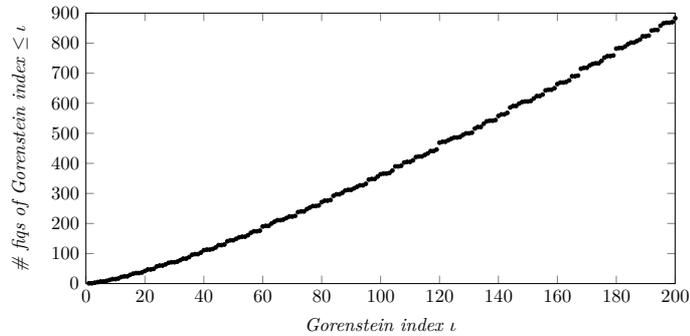
\begin{center}
\begin{tikzpicture}[scale=0.7]
\begin{axis}[
    xmin = 0, xmax = 200,
    ymin = 0, ymax = 900,
    width = \textwidth,
    height = 0.53\textwidth,
    xtick distance = 20,
    ytick distance = 100,
    xlabel = {Gorenstein index $\iota$},
    ylabel = {\# fiqs of Gorenstein index $\le \iota$}
]
% plot data line code
\addplot[mark size=1pt, only marks] table {plot-data-fiqs-rho-1.txt};
\end{axis}
\end{tikzpicture}
\end{center}
Exactly 150 full intrinsic quadric complex surfaces of
Picard number one and Gorenstein index at most 200 admit a 
K\"ahler-Einstein metric.

\goodbreak

\item
In Picard number two, we find in total 71$\,$198 full intrinsic
quadric surfaces of Gorenstein index at most 200, filtered
as follows:
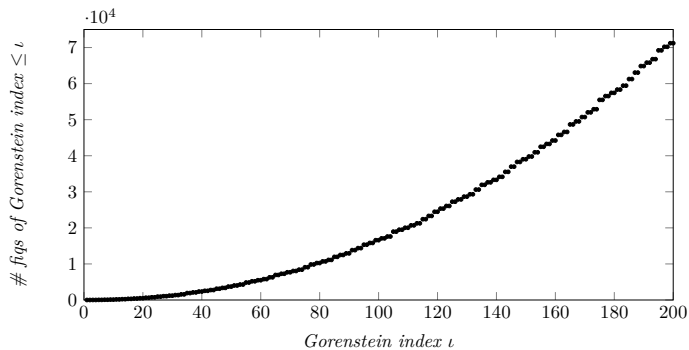
\begin{center}
\begin{tikzpicture}[scale=0.7]
\begin{axis}[
    xmin = 0, xmax = 200,
    ymin = 0, ymax = 75000,
    width = \textwidth,
    height = 0.53\textwidth,
    xtick distance = 20,
    ytick distance = 10000,
    xlabel = {Gorenstein index $\iota$},
    ylabel = {\# fiqs of Gorenstein index $\le \iota$}
]
% plot data line code
\addplot[mark size=1pt, only marks] table {plot-data-fiqs-rho-2.txt};
\end{axis}
\end{tikzpicture}
\end{center}
In Picard number two, there are no full intrinsic
quadric complex surfaces at all admitting a
K\"ahler-Einstein metric.

\item
In Picard number three, we find in total 15$\,$466$\,$258 full intrinsic
quadric surfaces of Gorenstein index at most 200, filtered
as follows:
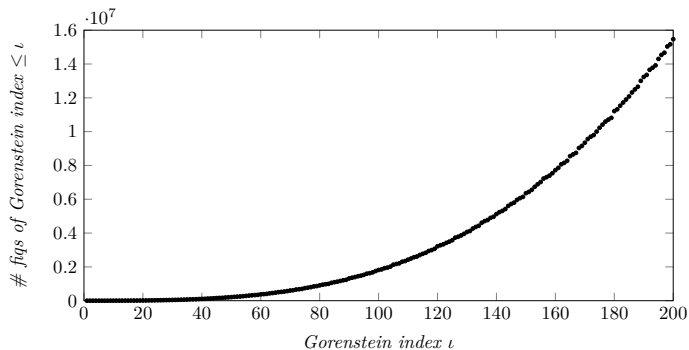
\begin{center}
\begin{tikzpicture}[scale=0.7]
\begin{axis}[
    xmin = 0, xmax = 200,
    ymin = 0, ymax = 16000000,
    width = \textwidth,
    height = 0.53\textwidth,
    xtick distance = 20,
    ytick distance = 2000000,
    xlabel = {Gorenstein index $\iota$},
    ylabel = {\# fiqs of Gorenstein index $\le \iota$}
]
% plot data line code
\addplot[mark size=1pt, only marks] table {plot-data-fiqs-rho-3.txt};
\end{axis}
\end{tikzpicture}
\end{center}
Exactly 1$\,$006$\,$633 full intrinsic quadric complex surfaces of
Picard number three and Gorenstein index at most 200
admit a K\"ahler-Einstein metric.
\end{enumerate}
\end{corollary}

We assume the reader to be familiar with the very basics of
toric geometry, in particular the construction of a
toric variety from its defining fan, the orbit decomposition,
toric divisors and homogeneous coordinates; see for
instance~\cites{Dan,Ful,CoLiSc}.
Theorems~\ref{thm:fiqs2matrix-rho1},~\ref{thm:fiqs2matrix-rho2}
and~\ref{thm:fiqs2matrix-rho3} are formulated and proven in
Sections~\ref{sec:rho-1},~\ref{sec:rho-2} and~\ref{sec:rho-3},
respectively.
The considerations follow a common pattern but differ 
in the details; for convenience, we present the complete
arguments in each case.
The geometric applications are given in Section~\ref{sec:geom-app}.
The defining matrices for the full intrinsic quadrics
are available under~\cite{LDPDB} for Gorenstein index
up to 200 in Picard numbers one, two
and for Gorenstein index up to 40 in Picard number
three.

\section{Full intrinsic quadric surfaces allow a $\KK^*$-action}

In this section we show that every full intrinsic
quadric surface is rational, $\QQ$-factorial,
projective and admits a non-trivial $\KK^*$-action.
This will allow us to work with the approach
to $\KK^*$-surfaces provided by~\cites{HaHe,HaSu};
see also~\cite{ArDeHaLa}*{Sec.~5.4}.
First let us give a precise definition
of a full intrinsic quadric; see also~\cite{BeHa}*{Sec.~9}.

\begin{definition}
A \emph{full intrinsic quadric} is a normal
complete variety $X$ with finitely generated
divisor class group and Cox ring of the form
\[
\mathcal{R}(X)
\ = \
\bigoplus_{\Cl(X)} \Gamma(X,\mathcal{O}(D))
\ = \ 
\KK[T_1, \ldots, T_n] / \bangle{g}
\]
with $\Cl(X)$-homogeneous generators
$T_1, \ldots, T_n \in \mathcal{R}(X)$
and a $\Cl(X)$-homogeneous
quadric $g \in \KK[T_1, \ldots, T_n]$
of full rank.
\end{definition}

\goodbreak

\begin{theorem}
\label{thm:fiqs-cr}
Let $X$ be a full intrinsic quadric surface.
Then $X$ is $\QQ$-factorial, rational, projective
and admits an effective $\KK^*$-action.
Moreover, the Picard number of $X$ satisfies 
$\rho(X) \le 3$ and its Cox ring allows a
$\Cl(X)$-graded presentation as
\[
\mathcal{R}(X)
\ \cong \
\begin{cases}
\KK[T_1, \ldots, T_4] / \bangle{T_1T_2+T_3^2+T_4^2},
&
\rho(X) = 1,
\\
\KK[T_1, \ldots, T_5] / \bangle{T_1T_2+T_3T_4+T_5^2},
&
\rho(X) = 2,
\\
\KK[T_1, \ldots, T_6] / \bangle{T_1T_2+T_3T_4+T_5T_6},
&
\rho(X) = 3.
\end{cases}
\]
\end{theorem}

\begin{proof}
By definition, $X$ is a normal complete surface with
finitely generated Cox ring.
From~\cite{ArDeHaLa}*{Thm.~4.3.3.5} we infer that $X$
is $\QQ$-factorial and projective.
Moreover, by~\cite{FaHa}*{Prop.~2.1}, we have a
$\Cl(X)$-graded presentation
\[
\mathcal{R}(X)
 \cong 
\KK[T_1, \ldots, T_{n+m}]
/
\bangle{g},
\quad
g  =  T_1T_2 + \ldots + T_{n-1}T_n +T_{n+1}^2 + \ldots + T_{n+m}^2.
\]

We use this to show $\rho(X) \le 3$.
Recall from~\cite{ArDeHaLa}*{Cor.~1.6.2.7 and Constr.~1.6.3.1}
that $X$ is the geometric
quotient of an open subset of the total coordinate space
\[
\bar X
\ = \
\Spec \, \mathcal{R}(X)
\ = \
V(g)
\ \subseteq \
\KK^{n+m}.
\]
by the quasitorus $H = \Spec \, \KK[\Cl(X)]$, which,
due to $\QQ$-factoriality of $X$, is of dimension~$\rho(X)$.
Consequently, we have
\[
2 \ = \ \dim(X) \ = \ \dim(\bar X) - \dim(H) \ = \ n+m-1-\rho(X).
\]
The degrees $w_1, \ldots, w_{n+m}$ of $T_1, \ldots, T_{n+m}$
generate $\Cl(X)$. Moreover, the degree $\mu = \deg(g)\in \Cl(X)$
satisfies $\mu = w_i+w_{i+1}$ for $i = 1,3, \ldots, n-1$
and $\mu = 2 w_{n+j}$ for $j = 1, \ldots, m$.
Thus, $\Cl(X)$ is generated by $w_1,w_2,w_3,w_5, \ldots, w_{n-1}$
and we see
\[
n+m-3 \ = \ \rho(X) \ \le \ 2 + \frac{n-2}{2}. 
\]
We conclude $n/2+m \le 4$ and thus $n \le 8$.
Assume $n=8$. Then $m=0$ and $\rho(X)=5$ hold.
Consequently, $(w_1,w_2,w_3,w_5,w_7)$ is a
basis for the rational vector space $\Cl_\QQ(X)$.
Let $u$ be a linear form on $\Cl_\QQ(X)$
such that
\[
\bangle{u,w_1}
=  
\bangle{u,w_2}
=  
\bangle{u,w_3}
=  
\bangle{u,w_5}
=
0,
\qquad
\bangle{u,w_7}
<
0.
\]
Then $u$ annihilates as well $\mu = w_1+w_2$ and thus
also $w_4 = \mu-w_3$ and $w_6=\mu-w_5$. Moreover,
$u$ evaluates positively on $w_8 = \mu - w_7$.
Consequently, computing the cone of movable divisor
classes according to~\cite{ArDeHaLa}*{Prop.~3.3.2.3},
we obtain
\[
\Mov(X)
\ = \
\bigcap_{i=1}^8 \cone(w_j; \ j \ne i)
\ \subseteq \
\ker(u)
\ \subseteq \
\Cl_\QQ(X).
\]
This contradicts to the fact that $\Mov(X)$ is
a cone of full dimension in $\Cl_\QQ(X)$.
We conclude, $n \le 6$. The case $n=6$, $m=1$ and
$\rho(X)=4$ is excluded by the same arguments as used for
$n=8$ and $\rho(X)=5$.

Thus, we have $\rho(X) \le 3$. If $\rho(X)=3$, then
$n=6$ and $m=0$, which leads to third case
in the assertion.
For $\rho(X)=2$, we are left with the choices
$n=4$ with $m=1$ and $n=0,2$. The first one
gives the second case of the assertion
and the other two would produce, similarly
as before, a cone of movable divisors
of dimension less than that of  $\Cl_\QQ(X)$
and hence can't occur.

For $\rho(X)=1$, we find the possibility
$n=m=2$, which is the first case of the
assertion.
Also, $n=0,4$ might happen.
We first exclude the case $n=4$.
There, the prospective total coordinate
space $\bar X = \Spec \, \KK[\mathcal{R}(X)]$
is explicitly given as
\[
\bar X
\ = \
V(T_1T_2+T_3T_4)
\ \subseteq \
\KK^4.
\]
In this setting, we find a diagonal action
of a three-dimensional torus $\TT$
on $\KK^4$ turning $\bar X$
into a toric variety.
Thus, $X$ as a GIT-quotient of
$\bar X$ by a one-dimensional subgroup
of $\TT$ is as well a toric variety
and must have a polynomial ring
as its Cox ring; a contradiction
to $\bar X$ being singular.

Now we exclude the case $n=0$. 
Due to $\QQ$-factoriality of $X$,
the divisor class group $\Cl(X)$
is of rank one and hence
of the form $\ZZ \otimes \Gamma$
with a finite abelian group~$\Gamma$.
Moreover, the prospective Cox ring of $X$
is given as
\[
\mathcal{R}(X)
\ = \
\KK[T_1,T_2,T_3,T_4] / \bangle{T_1^2+T_2^2+T_3^2+T_4^2}.
\]
With respect to the $\Cl(X)$-grading of $\mathcal{R}(X)$,
the generators $T_i$ as well as the
relation~$g$ are homogeneous.
Moreover, the $T_i$ are prime in $\mathcal{R}(X)$
and any three of the $w_i = \deg(T_i)$ must
generate $\Cl(X)$;
see~\cite{ArDeHaLa}*{Def.~3.2.1.1 and Cor.~3.2.1.11}.
This is realized for instance by  
\[
\Cl(X) = \ZZ \otimes \ZZ / 2\ZZ \otimes  \ZZ / 2\ZZ,
\qquad
Q = \left[
\begin{array}{cccc}  
1 & 1 & 1 & 1
\\
\bar 0 & \bar 1 & \bar 1 & \bar 0 
\\
\bar 0 & \bar 1 & \bar 0 & \bar 1 
\end{array}
\right],
\]
where $Q$ hosts the $w_i$ as its columns.
Note that $w_1,w_2,w_3,w_4$ are pairwise different,
which ensures that $\mathcal{R}(X)$ is not isomorphic
to one of the graded rings with $n=2,4$.
One directly verifies that $\Cl(X)$ and $\mathcal{R}(X)$
as above are uniquely determined by these features.

The task is to show that the above $\Cl(X)$-graded algebra
$\mathcal{R}(X)$ can't be a Cox ring.
Otherwise, as $\bar X$ is smooth apart from the origin,
$X$ would be quasismooth, hence log terminal.
Moreover, we can apply~\cite{ArDeHaLa}*{Cor.~3.3.3.3}
to see that $X$ is a del Pezzo surface of Picard number
one and Gorenstein index one; we used the software
package~\cite{HaKe} for the computation.
The Cox rings of all log del Pezzo surfaces of
Picard number one and Gorenstein index one without
torus action have been computed in~\cite{HaKeLa}*{Thm.~4.1}
and for those with torus action the Cox rings are listed
in~\cite{ArDeHaLa}*{5.4.4.2}; none of these Cox rings
is isomorphic to $\mathcal{R}(X)$ from above.

We verified, that the Cox ring of any full intrinsic
quadric sruface $X$ is as in the assertion, in
particular it is defined by trinomial relations.
Consequently, the associated total coordinate space
$\bar X$ allows a diagonal torus action
of complexity one.
This action induces a non-trivial $\KK^*$-action
on $X$.
Since $\Cl(X)$ is finitely generated by assumption,
this forces $X$ to be rational.
\end{proof}

\section{Picard number one}
\label{sec:rho-1}

The main result of this section,
Theorem~\ref{thm:fiqs2matrix-rho1},
provides the description of all
full intrinsic quadric surfaces
of Picard number one in terms of
the local Gorenstein indices of
two of their possibly singular
points.

\begin{construction}[Full intrinsic quadric
surfaces $X$ of Picard number one
as $\KK^*$-surfaces]
\label{constr:matrix2fiqs-rho1}
Consider an integral matrix of the
form
\[
P
\ := \ 
[v_1,v_2,v_3,v_4]
\ := \ 
\left[
\begin{array}{cccc}
-1 & -1 & 2 & 0  
\\
-1 & -1 & 0 & 2
\\
 a & b & 1 & 1
\end{array}
\right],
\qquad
b \le -2, \ 0 \le a \le -b -2.
\]
Let $Z$ be the toric variety arising from the
fan $\Sigma$ in $\ZZ^3$ with generator matrix $P$
and the maximal cones
\[
\sigma^+ := \cone(v_1,v_3,v_4),
\qquad
\sigma^- := \cone(v_2,v_3,v_4),
\qquad
\tau_0 := \cone(v_1,v_2).
\]
Denote by $U_1,U_2,U_3$ the coordinate
functions on the standard 3-torus $\TT^3 \subseteq Z$.
Then we obtain a normal, non-toric, rational, projective
surface
\[
X \ := \ X(P) \ := \ \overline{V(h)} \ \subseteq \ Z,
\qquad
h \ := \ 1 + U_1 + U_2 \ \in \mathcal{O}(\TT^3),
\]
with a $\KK^*$-action, given on $V(h) \subseteq \TT^3$
by $t \cdot x = (x_1, x_2, tx_3)$.
With $K := \ZZ^4/\im(P^*)$, the divisor class group
of $X$ is given by 
\[
\Cl(X)
\ \cong \
\Cl(Z)
\ \cong \
K
\ \cong \
\ZZ \times \ZZ \, / \, 2\gcd(2a+2, \, a-b) \ZZ \, .
\]
Moreover, denoting by $Q \colon\ZZ^4 \to K$ the projection,
we obtain the following description of the Cox ring of $X$
as a graded algebra:
\[
\mathcal{R}(X) 
\ \cong \   
\KK[T_1, \ldots, T_4] / \bangle{T_1T_2+T_3^2+T_4^2},
\qquad
\deg(T_i) = Q(e_i) = [D_i],
\]
where $D_i \subseteq X$ is the prime divisor on $X$ obtained by
intersecting $X$ with the toric prime divisor of $Z$
given by the ray through $v_i$ and $[D_i] \in \Cl(X)$ denotes
its class.
\end{construction}

\begin{proof}
By definition, the columns $v_1,\ldots,v_4$ of $P$ are pairwise
different primitive integral vectors. Moreover, $v_1,\ldots,v_4$
generate $\QQ^3$ as a convex cone, as we have
\[
2v_1 + v_3 + v_4 =  [0,0,2a+2], \ a \ge 0,
\qquad
2v_2 + v_3 + v_4 =  [0,0,2b+2], \ b \le -2.
\]
Consequently, $P$ is a defining matrix of a rational projective
$\KK^*$-surface in the sense of~\cite{ArDeHaLa}*{Constr.~3.4.3.1}.
The statements follow from those
of~\cite{ArDeHaLa}*{Sec.~3.4.3}.
% see also~\cite{HaHaSp}*{Sec.~5.2}.
\end{proof}

The \emph{local class group} $\Cl(X,x)$ of a point $x \in X$
is the group of Weil divisors of~$X$ modulo those being
principal near $x$, and by $\cl(X,x)$ the order of $\Cl(X,x)$.

\begin{proposition}
\label{prop:fiqs-rho1-props-1}
Let $X = X(P)$ arise from
Construction~\ref{constr:matrix2fiqs-rho1}.
The fixed points of the $\KK^*$-action on $X$ are
given in Cox coordinates by 
\[
x^+ \ := \ [0,1,0,0],
\qquad
x^- \ := \ [1,0,0,0],
\qquad
x_0 \ := \ [0,0,1,I].
\]
Moreover, for the orders of the local class groups of
the fixed points of the $\KK^*$-action we obtain
\[
\cl(X,x^+) = 4a+4,
\qquad
\cl(X,x^-) = -4b-4,
\qquad
\cl(X,x_0) = a-b.
\]
Finally, the ordered pair $(4a+4, \, -4-4b)$
is an isomorphy invariant of the algebraic
surface $X$.
\end{proposition}

\begin{proof}
For the first statement, we refer to~\cite{HaHaSp}*{Rem.~5.6}.
For the second one, we
use~\cite{ArDeHaLa}*{Lemma~2.1.4.1, Prop.~3.3.1.5}
to compute the
local class group orders of $x^+$, $x^-$ and $x_0$ as
\[
\det[v_1,v_3,v_4],
\qquad
\det[v_2,v_3,v_4],
\qquad
\det
\left[\renewcommand{\arraystretch}{.7} \arraycolsep=.7\arraycolsep 
  \begin{array}{rr}
  -1 &  -1 \\ a &  b 
  \end{array}
\right].
\] 
For the last statement recall from~\cite{ArDeHaLa}*{Prop.~5.4.1.9}
that $x^+$, $x^-$ are the only $\KK^*$-fixed points
lying in the closure of infinitely many orbits.
Thus, $\{\cl(X,x^+),\cl(X,x^-)\}$ and
$\cl(X,x_0)$ are invariants of the $\KK^*$-surface $X$.
Since on a non-toric, rational, projective surface
any two $\KK^*$-actions are conjugate in the automorphism
group, the assertion follows.
\end{proof}

\begin{proposition}
\label{prop:fiqs2matrix-rho1}
Every full intrinsic quadric surface $X$ 
of Picard number one is isomorphic to
an $X(P)$ for precisely one matrix $P$
from Construction~\ref{constr:matrix2fiqs-rho1}.
\end{proposition}

\begin{proof}
According to Theorem~\ref{thm:fiqs-cr}
and~\cite{HaHiWr}*{Ex.~7.1},
the  defining matrix $P$ is of the format
$3 \times 4$ and the first two rows are as
in the assertion:
\[
P
\ = \ 
\left[
\begin{array}{cccc}
-1 & -1 & 2 & 0 
\\
-1 & -1 & 0 & 2
\\
d_1 & d_2 & d_3 & d_4 
\end{array}
\right].
\]
Admissible operations~\cite{HaHaSp}*{Prop.~6.7}
bring us to the setting of
Construction~\ref{constr:matrix2fiqs-rho1}.
First, adding suitable multiples of the first two rows
to the last one yields
\[
P
\ = \ 
\left[
\begin{array}{cccc}
-1 & -1 & 2 & 0 
\\
-1 & -1 & 0 & 2
\\
a & b & 1 & 1
\end{array}
\right].
\]
Second, swapping the first two columns if necessary,
we achieve that $P$ is slope-ordered, meaning
\[
a > b.
\]
As for any defining matrix of a
rational $\KK^*$-surface with two elliptic fixed
points, slope orderedness implies 
\[
a + \frac{1}{2} + \frac{1}{2} =: m^+ > 0,
\qquad\qquad
b + \frac{1}{2} + \frac{1}{2} =: m^- < 0,
\]
see~\cite{HaHaSp}*{Rem.~7.5}.
Multiplying the last row by $-1$ turns $m^{\pm}$
into $m^{\mp}$.
Doing so, if necessary, and re-arranging
via the first two steps yields
\[
a + 1 
\ \le \
-b -1.
\]
If $X(P) \cong X(P')$ holds with $P,P'$ as in
Construction~\ref{constr:matrix2fiqs-rho1},
then we have $P =P'$, as due to 
Proposition~\ref{prop:fiqs-rho1-props-1},
the entries $a,b$ of $P$ and $a',b'$ of $P'$
satisfy 
\[
(4a+4, \, -4b-4)
\ = \
(4a'+4, \, -4b'- 4).
\]
\end{proof}

Recall that the \emph{Gorenstein index} of a $\QQ$-factorial variety
$X$ is the smallest positive integer~$\iota_X$ such that the
$\iota_X$-fold of a canonical divisor of $X$ is Cartier.
The  \emph{local Gorenstein index} $\iota_x$ of a point $x \in X$
is the smallest positive integer such that the $\iota_x$-fold
of a canonical divisor of $X$ is Cartier near $x$.

\begin{theorem}
\label{thm:fiqs2matrix-rho1} 
For any $\iota \in \ZZ_{\ge 1}$, consider the set~$M_\iota$
of pairs $\eta = (\iota^+,\iota^-) \in \ZZ_{\ge 1}^2$
with $\lcm(\iota^+,\iota^-) = \iota$. 
Define subsets
\[
\begin{array}{lcl}
S_{11}(1,\iota)
& := & 
\{
\eta \in M_\iota;
\ \iota^+ \text{ odd},
\ \iota^-  \text{ odd},
\ \iota^+ \le \iota^-
\},
\\[5pt]       
S_{12}(1,\iota)
& := & 
\{
\eta \in M_\iota;
\ \iota^+ \text{ odd},
\ \iota^- \text{ even},
\ 4 \mid \iota^-,       
\ 2\iota^+ \le \iota^- 
\},
\\[5pt]       
S_{21}(1,\iota)
& := & 
\{
\eta \in M_\iota;
\ \iota^+ \text{ even},
\ \iota^- \text{ odd},
\ 4 \mid \iota^+,       
\ \iota^+ \le 2\iota^-
\},
\\[5pt]       
S_{22}(1,\iota)
& := & 
\{
\eta \in M_\iota;
\ \iota^+ \text{ even},
\ \iota^- \text{ even},
\ 4 \mid \iota^+,       
\ 4 \mid \iota^-,       
\ \iota^+ \le \iota^- 
\}.
\end{array}
\]
Then each set $S_{ij}(1,\iota)$ provides us with a series
of defining matrices $P_\eta$ of full intrinsic quadric
surfaces:

{\small
\[
\begin{array}{lcl}
\eta = (\iota^+,\iota^-) \in S_{11}(1,\iota)\colon
& &
\eta = (\iota^+,\iota^-) \in S_{12}(1,\iota)\colon
\\[5pt]
P_\eta = 
\left[
\begin{array}{cccc}
-1 & -1 & 2 & 0 
\\
-1 & -1 & 0 & 2
\\
\iota^+-1 & -\iota^--1 & 1 & 1 
\end{array}
\right],
&  &
P_\eta  =
\left[
\begin{array}{cccc}
-1 & -1 & 2 & 0 
\\
-1 & -1 & 0 & 2 
\\
\iota^+-1 & -\frac{\iota^-}{2}-1 & 1 & 1 
\end{array}
\right],
\\[30pt]
\eta = (\iota^+,\iota^-) \in S_{21}(1,\iota)\colon
& &
\eta = (\iota^+,\iota^-) \in S_{22}(1,\iota)\colon
\\[5pt]
P_\eta  =
\left[
\begin{array}{cccc}
-1 & -1 & 2 & 0 
\\
-1 & -1 & 0 & 2
\\
\frac{\iota^+}{2}-1 & -\iota^- -1 & 1 & 1 
\end{array}
\right],
&  &
P_\eta  =
\left[
\begin{array}{cccc}
-1 & -1 & 2 & 0 
\\
-1 & -1 & 0 & 2
\\
\frac{\iota^+}{2}-1 & -\frac{\iota^-}{2}-1 & 1 & 1 
\end{array}
\right].
\end{array}           
\]
}

\medskip

\noindent
Each surface $X(P_\eta)$ is of Picard number one, Gorenstein index
$\iota = \lcm(\iota^+,\iota^-)$ and $\iota^\pm$
are the local Gorenstein indices of the points
\[
x^+ \ = \ [0,1,0,0], \qquad\qquad
x^- \ = \ [1,0,0,0].
\]
Moreover, every full intrinsic quadric surface of Picard number one
and Gorenstein index~$\iota$ is isomorphic to $X(P_\eta)$ for
precisely one $P_\eta$ from the above list.
\end{theorem}

\goodbreak

\begin{proof}
Let $X$ be a full intrinsic quadric surface of Picard number one.
We first show $X \cong X(P_\eta)$ with $P_\eta$ from
the above list and check the local Gorenstein indices.
Proposition~\ref{prop:fiqs2matrix-rho1} allows us to assume
$X = X(P)$ with a unique $P$ of the form
\[
P
\ = \ 
\left[
\begin{array}{cccc}
-1 & -1 & 2 & 0  
\\
-1 & -1 & 0 & 2
\\
 a & b & 1 & 1
\end{array}
\right],
\qquad
b \le -2, \ 0 \le a \le -b -2.
\]
According to~\cite{ArDeHaLa}*{Prop.~3.3.3.2}, we have
the anticanonical divisor $-\mathcal{K} = D_3+D_4$
on $X$ and~\cite{HaHaSp}*{Prop.~8.9} tells us
that the linear forms $u^\pm$ representing the $\iota^\pm$-fold
of $-\mathcal{K}$ near $x^\pm$ are given by

{\small
\[
u^+
\ = \ 
\left[
\frac{a\iota^+}{2a+2}, \,  \frac{a\iota^+}{2a+2}, \,  \frac{\iota^+}{a+1},
\right],
\qquad
u^-
\ = \ 
\left[
\frac{b\iota^-}{2b+2}, \,  \frac{b\iota^-}{2b+2}, \,  \frac{\iota^-}{b+1} 
\right].
\]
}

\noindent
By the definition of the local Gorenstein index, these
are primitive integral vectors.
Consequently, the local Gorenstein indices $\iota^\pm$ of
$x^\pm$ are 
\[
\iota^+
\ = \
\begin{cases}
a+1, & a \text{ even},
\\
2a+2,   & a \text{ odd},
\end{cases}
\qquad\qquad
\iota^-
\ = \
\begin{cases}
-b-1, & b \text{ even},
\\
-2b-2,   & b \text{ odd}.
\end{cases}
\]
In particular, $\iota^+$, $\iota^-$ is even (odd) if and only if
$a$, $b$ is odd (even), respectively.
Moreover, if $\iota^\pm$ is even, then it is divisible by four.
Thus, $P$ is one of the matrices~$P_\eta$ with $\eta=(\iota^+,\iota^-)$
listed in the assertion and $\iota^\pm$ is the local Gorenstein
index of~$x^\pm$.

Conversely, all the matrices $P_\eta$ listed in the assertion fit
into the shape of Construction~\ref{constr:matrix2fiqs-rho1}
and thus deliver full intrinsic quadric surfaces $X = X(P_\eta)$.
By~\cite{HaHaSp}*{Prop.~8.8}, the point
$x_0 = [0,0,1,I] \in X$ has local Gorenstein
index one, hence the resulting $X$ is of Gorenstein index
$\iota = \lcm(\iota^+, \iota^-)$.

Finally, we ensure that the matrices $P_\eta$ 
listed in the assertion define pairwise
non-isomorphic $X(P_\eta)$.
By Proposition~\ref{prop:fiqs2matrix-rho1},
this means to show that any two matrices arising
from different $S_{ij}(1,\iota)$ differ from each
other.
This is done by comparing the parity vectors
$(\bar a, \bar b)$ in $\ZZ/2\ZZ \times \ZZ/2\ZZ$
of the first two entries $a,b$ of the third row
of $P_\eta$ for the $\eta \in S_{ij}(1,\iota)$:
\renewcommand{\arraystretch}{1.5}
\begin{center}
\begin{tabular}{c|c|c|c|c}
&
$S_{11}(1,\iota)$
&
$S_{12}(1,\iota)$
&
$S_{21}(1,\iota)$
&
$S_{22}(1,\iota)$
\\
\hline
$(\bar a, \bar b)$
&
$(\bar 0, \bar 0)$
&
$(\bar 0, \bar 1)$
&
$(\bar 1, \bar 0)$
&
$(\bar 1, \bar 1)$
\end{tabular}
\end{center}
\end{proof}

\section{Picard number two}
\label{sec:rho-2}

The main result of this section,
Theorem~\ref{thm:fiqs2matrix-rho2},
provides the description of all
full intrinsic quadric surfaces
of Picard number two in terms of
the local Gorenstein indices of
two of their possibly singular
and the local class group order
of another possibly singular
point.

\begin{construction}
[Full intrinsic quadric
surfaces $X$ of Picard number two
as $\KK^*$-surfaces]
\label{constr:matrix2fiqs-rho2}
Consider an integral matrix of the
form
\[
P
\ := \ 
[v_1,v_2,v_3,v_4,v_5]
\ := \ 
\left[
\begin{array}{ccccc}
-1 & -1 & 1 & 1 & 0 
\\
-1 & -1 & 0 & 0 & 2
\\
a & b & 0 & c & 1
\end{array}
\right],
\quad
\begin{array}{l}
b < a, \ c < 0, \ a \ge 0,
\\[5pt]
b+c \le -1, \ a-b \le -c, 
\\[5pt]
a \le -b-c-1.
\end{array}
\]
Let $Z$ be the toric variety arising from the fan 
$\Sigma$ in $\ZZ^3$ with generator matrix $P$
and the maximal cones 
\[
\sigma^+ := \cone(v_1,v_3,v_5),
\qquad
\sigma^- := \cone(v_2,v_4,v_5),
\]
\[
\tau_0 := \cone(v_1,v_2),
\qquad
\tau_1 := \cone(v_3,v_4).
\]
Denote by $U_1,U_2,U_3$ the coordinate
functions on the standard 3-torus $\TT^3 \subseteq Z$.
Then we obtain a normal, non-toric, rational, projective surface
\[
X \ := \ X(P) \ := \ \overline{V(h)} \ \subseteq \ Z,
\qquad
h \ := \ 1 + U_1 + U_2 \ \in \mathcal{O}(\TT^3),
\]
with a $\KK^*$-action, given on $V(h) \subseteq \TT^3$
by $t \cdot x = (x_1, x_2, tx_3)$.
With $K := \ZZ^5/\im(P^*)$, the divisor class group
of $X$ is given as
\[
\Cl(X) \ \cong \ K \ \cong \ \Cl(Z)
\ \cong \ 
\ZZ^2 \times \ZZ \, / \gcd(2a+1,a-b,-c) \ZZ.
\]
Moreover, denoting by $Q \colon\ZZ^5 \to K$ the projection,
we obtain the following description of Cox ring of $X$
as graded algebra:
\[
\mathcal{R}(X) 
\ \cong \   
\KK[T_1, \ldots, T_5] / \bangle{T_1T_2+T_3T_4+T_5^2},
\qquad
\deg(T_i) = Q(e_i) = [D_i],
\]
where $D_i \subseteq X$ is the prime divisor on $X$ obtained by
intersecting $X$ with the toric prime divisor of $Z$
given by the ray through $v_i$ and $[D_i] \in \Cl(X)$ denotes
its class.
\end{construction}

\begin{proof}
By definition, the columns $v_1,\ldots,v_5$ of $P$ are pairwise
different primitive integral vectors. Moreover, $v_1,\ldots,v_5$
generate $\QQ^3$ as a convex cone, as we have
\[
2v_1 + 2v_3 + v_5 =  [0,0,2a+1], \ a \ge 0,
\]
\[
2v_2 + 2v_4 + v_5 =  [0,0,2b+2c+1], \ b+c \le -1.
\]
Consequently, $P$ is a defining matrix of a rational projective
$\KK^*$-surface in the sense of~\cite{ArDeHaLa}*{Constr.~3.4.3.1}.
The statements follow from those
of~\cite{ArDeHaLa}*{Sec.~3.4.3}.
\end{proof}

\begin{proposition}
\label{prop:fiqs-rho2-props-1}
Let $X = X(P)$ arise from
Construction~\ref{constr:matrix2fiqs-rho2}.
The fixed points of the $\KK^*$-action on $X$ are
given in Cox coordinates by 
\[
x^+ \ := \ [0,1,0,1,0],
\qquad
x^- \ := \ [1,0,1,0,0],
\quad
\]
\[
x_0 \ := \ [0,0,1,1,I],
\qquad
x_1 \ := \ [1,1,0,0,I].
\]
Moreover, the orders of the local class groups of the
fixed points of the $\KK^*$-action on $X$ are given by 
\[
\cl(X,x^+) = 1+2a,
\qquad
\cl(X,x^-) = -1-2b-2c,
\]
\[
\cl(X,x_0) = a-b,
\qquad
\cl(X,x_1) = -c.
\]
Finally, the ordered pairs $(1+2a, \, -1-2b-2c)$
and $(a-b, \, -c)$ are isomorphy invariants of the
algebraic surface $X$.
\end{proposition}

\begin{proof}
The same references as in the proof of
Proposition~\ref{prop:fiqs-rho1-props-1},
provide us with the description of the
fixed points and show that the local class
group orders of
$x^+$, $x^-$, $x_0$ and $x_1$ compute
as
\[
\det[v_1,v_3,v_5],
\qquad
\det[v_2,v_4,v_5],
\qquad
\det
\left[\renewcommand{\arraystretch}{.7} \arraycolsep=.7\arraycolsep 
  \begin{array}{rr}
  -1 &  -1 \\ a &  b 
  \end{array}
\right],
\qquad
-\det
\left[\renewcommand{\arraystretch}{.7} \arraycolsep=.7\arraycolsep 
  \begin{array}{rr}
  1 &  1 \\ 0 &  c
  \end{array}
\right].
\]
As mentioned in the proof of Proposition~\ref{prop:fiqs-rho1-props-1},
the fixed points $x^+$, $x^-$ are the only ones
lying in the closure of infinitely many orbits.
Moreover, each of the remaining two fixed
points $x_0$, $x_1$ lies in the closure of precisely
two non-trivial orbits.
Thus, $\{\cl(X,x^+),\cl(X,x^-)\}$ as well as
$\{\cl(X,x_0),\cl(X,x_1)\}$ are invariants of
the $\KK^*$-surface $X$.
As before, the assertion follows from the fact
that on a non-toric, rational, projective
surface any two $\KK^*$-actions
are conjugate in the automorphism group.
\end{proof}

\goodbreak

\begin{proposition}
\label{prop:fiqs2matrix-rho2}
Every full intrinsic quadric surface $X$ 
of Picard number two is isomorphic to
an $X(P)$ for precisely one matrix $P$
from Construction~\ref{constr:matrix2fiqs-rho2}.
\end{proposition}

\begin{proof}
Using again Theorem~\ref{thm:fiqs-cr}
and~\cite{HaHiWr}*{Ex.~7.1},
we see that the defining matrix~$P$
is of the format $3 \times 5$ and the
first two rows look as in the assertion:
\[
P
\ = \ 
\left[
\begin{array}{ccccc}
-1 & -1 & 1 & 1 & 0 
\\
-1 & -1 & 0 & 0 & 2 
\\
d_1 & d_2 & d_3 & d_4 & d_5
\end{array}
\right].
\]
We achieve the desired shape of $P$
by admissible operations~\cite{HaHaSp}*{Prop.~6.7}.
First, adding suitable multiples of the first two rows
to the last one yields
\[
P
\ = \ 
\left[
\begin{array}{ccccc}
-1 & -1 & 1 & 1 & 0 
\\
-1 & -1 & 0 & 0 & 2
\\
a & b & 0 & c & 1 
\end{array}
\right].
\]
Second, swapping the columns $v_1$ and $v_2$ as well as $v_3$ and $v_4$
if neccesary and re-arranging via the
first step, we achieve that $P$ is slope-ordered, meaning
\[
a > b, \qquad 0 > c.
\]
Third, swapping the first two columns blocks,
that means $[v_1,v_2]$ and $[v_3,v_4]$,
if neccesary and re-adjusting the entries,
we can ensure
\[
a - b  \le - c.
\]
As for any defining matrix of a rational $\KK^*$-surface
with two elliptic fixed points, slope orderedness implies 
\[
a+\frac{1}{2} =: m^+ > 0,
\qquad\qquad
b+c+\frac{1}{2} =: m^- < 0.
\]  
Multiplying the last row by $-1$ turns $m^{\pm}$
into $m^{\mp}$. Doing so, if necessary, and re-arranging
via the first two steps yields
\[
a
\ \le \
-b-c-1.
\]
We show that $X(P) \cong X(P')$ with matrices $P$
and $P'$ as in Construction~\ref{constr:matrix2fiqs-rho2}
implies $P = P'$.
Proposition~\ref{prop:fiqs-rho2-props-1}
yields equality of the ordered tuples 
\[
(1+2a, \, -1-2b-2c)
 = 
(1+2a', \, -1-2b'-2c'),
\quad
(a-b, \, -c)
 = 
(a'-b', -c')
\]
built from the entries of the third row of
$P$ and $P'$ respectively.
From this we directly derive $P = P'$.
\end{proof}

\begin{theorem}
\label{thm:fiqs2matrix-rho2}
For any $\iota \in \ZZ_{\ge 1}$, consider the set~$M_\iota$
of triples
$\eta = (\iota^+,\iota^-,c)$, where
$\iota^+,\iota^- \in \ZZ_{\ge 1}$ with
$\lcm(\iota^+,\iota^-) = \iota$
and $c \in \ZZ_{\le -1}$.
Define subsets
\[
\arraycolsep=.7\arraycolsep
\begin{array}{lcl}
S_{11}(2,\iota)
& := & 
\{
\eta \in M_\iota;
\ 2 \nmid \iota^+,\iota^-,
\ 3 \nmid \iota^+,\iota^-,
\ \iota^+ \le \iota^-,
\ 1-\frac{\iota^++\iota^-}{2} \le c \le -\frac{\iota^++\iota^-}{4}
\},
\\[7pt]       
S_{12}(2,\iota)
& := & 
\{
\eta \in M_\iota;
\ 2 \nmid \iota^+,\iota^-,
\ 3 \nmid \iota^+,
\ \iota^+ \le 3\iota^-,
\ 1-\frac{\iota^++3\iota^-}{2} \le c \le -\frac{\iota^++3\iota^-}{4}
\},
\\[7pt]       
S_{21}(2,\iota)
& := & 
\{
\eta \in M_\iota;
\ 2 \nmid \iota^+,\iota^-,
\ 3 \nmid \iota^-,       
\ 3\iota^+ \le \iota^-,
\ 1-\frac{3\iota^++\iota^-}{2} \le c \le -\frac{3\iota^++\iota^-}{4}
\},
\\[7pt]       
S_{22}(2,\iota)
& := & 
\{
\eta \in M_\iota;
\ 2 \nmid \iota^+,\iota^-,
\ \iota^+ \le \iota^-,
\ 1-\frac{3\iota^++3\iota^-}{2} \le c \le -\frac{3\iota^++3\iota^-}{4}
\}.
\end{array}
\]

\medskip

\noindent
Then each set $S_{ij}(2,\iota)$ provides us with a series
of defining matrices $P_\eta$ of full intrinsic quadric
surfaces:

{\small
\[
\arraycolsep=.8\arraycolsep
\begin{array}{lcl}
\eta = (\iota^+,\iota^-,c) \in S_{11}(2,\iota)\colon
& &
\eta = (\iota^+,\iota^-,c) \in S_{12}(2,\iota)\colon
\\[5pt]
P_\eta =
\left[
\begin{array}{ccccc}
-1 & -1 & 1 & 1 & 0 
\\
-1 & -1 & 0 & 0 & 2
\\
\frac{\iota^+-1}{2} & -\frac{\iota^-+1}{2} -c & 0 & c & 1 
\end{array}
\right],
&  &
P_\eta  =
\left[
\begin{array}{ccccc}
-1 & -1 & 1 & 1 & 0 
\\
-1 & -1 & 0 & 0 & 2
\\
\frac{\iota^+-1}{2} & -\frac{3\iota^-+1}{2} -c & 0 & c & 1
\end{array}
\right],
\\[30pt]
\eta = (\iota^+,\iota^-,c) \in S_{21}(2,\iota)\colon
& &
\eta = (\iota^+,\iota^-,c) \in S_{22}(2,\iota)\colon
\\[5pt]
P_\eta  =
\left[
\begin{array}{ccccc}
-1 & -1 & 1 & 1 & 0 
\\
-1 & -1 & 0 & 0 & 2
\\
\frac{3\iota^+-1}{2}  & -\frac{\iota^-+1}{2} -c  & 0 & c & 1
\end{array}
\right],
&  &
P_\eta  =
\left[
\begin{array}{ccccc}
-1 & -1 & 1 & 1 & 0 
\\
-1 & -1 & 0 & 0 & 2
\\
\frac{3\iota^+-1}{2} &  -\frac{3\iota^-+1}{2} -c & 0 & c & 1 
\end{array}
\right].
\end{array}           
\]
}
% {\small
% \[
% \begin{array}{ll}
% P_\eta =
% \left[
% \begin{array}{ccccc}
% -1 & -1 & 1 & 1 & 0 
% \\
% -1 & -1 & 0 & 0 & 2
% \\
% \frac{\iota^+-1}{2} & -\frac{\iota^-+1}{2} -c & 0 & c & 1 
% \end{array}
% \right],
% &
% \eta \in S_{11}(2,\iota),
% \\[24pt]
% P_\eta =
% \left[
% \begin{array}{ccccc}
% -1 & -1 & 1 & 1 & 0 
% \\
% -1 & -1 & 0 & 0 & 2
% \\
% \frac{\iota^+-1}{2} & -\frac{3\iota^-+1}{2} -c & 0 & c & 1
% \end{array}
% \right],
% &
% \eta \in S_{12}(2,\iota),
% \\[24pt]
% P_\eta = 
% \left[
% \begin{array}{ccccc}
% -1 & -1 & 1 & 1 & 0 
% \\
% -1 & -1 & 0 & 0 & 2
% \\
% \frac{3\iota^+-1}{2}  & -\frac{\iota^-+1}{2} -c  & 0 & c & 1
% \end{array}
% \right],
% &
% \eta \in S_{21}(2,\iota),
% \\[24pt]
% P_\eta =  
% \left[
% \begin{array}{ccccc}
% -1 & -1 & 1 & 1 & 0 
% \\
% -1 & -1 & 0 & 0 & 2
% \\
% \frac{3\iota^+-1}{2} &  -\frac{3\iota^-+1}{2} -c & 0 & c & 1 
% \end{array}
% \right],
% &
% \eta \in S_{22}(2,\iota).
% \end{array}           
% \]
% }

\medskip

\noindent
Each surface $X(P_\eta)$ is of Picard number two,
Gorenstein index $\iota = \lcm(\iota^+,\iota^-)$
and $\iota^+$, $\iota^-$ resp. $-c$ are the local
Gorenstein indices resp. local class group order of 
\[
x^+ \ = \ [0,1,0,1,0], \qquad
x^- \ = \ [1,0,1,0,0], \qquad
x_1 \ = \ [1,1,0,0,I].
\]
Finally, every full intrinsic quadric surface of Picard number two
and Gorenstein index~$\iota$ is isomorphic to $X(P_\eta)$ for
precisely one $P_\eta$ from the above list.
\end{theorem}

\goodbreak

\begin{proof}
Let $X$ be a full intrinsic quadric surface
of Picard number two.  
Then Proposition~\ref{prop:fiqs2matrix-rho2}
allows us to assume $X = X(P)$ with 
\[
P
\ = \ 
\left[
\begin{array}{ccccc}
-1 & -1 & 1 & 1 & 0 
\\
-1 & -1 & 0 & 0 & 2
\\
a & b & 0 & c & 1
\end{array}
\right],
\quad
\begin{array}{l}
b < a, \ c < 0, \ a \ge 0,
\\[5pt]
b+c \le -1, \ a-b \le -c, 
\\[5pt]
a \le -b-c-1.
\end{array}
\]
Consider the anticanonical divisor $-\mathcal{K}=D_3+D_4+D_5$
on $X(P)$.
The linear forms~$u^\pm$ representing the $\iota^\pm$-fold
of $-\mathcal{K}$ near $x^\pm$ are given by
{\small
\[
u^+
 = 
\left[
\iota^+, \,  \frac{(a-1)\iota^+}{1+2a}, \,  \frac{3\iota^+}{1+2a}
\right],
\quad
u^-
 = 
\left[
\frac{(2b-c+1)\iota^-}{2b + 2c + 1}, \,  \frac{(b+c-1)\iota^-}{2b+2c+1}, \,  -\frac{3\iota^-}{2b+2c+1} 
\right].
\]
}

\noindent
By the definition of the local Gorenstein index, these 
are primitive integral vectors.
Together with the fact that $\iota^\pm$ divides
$\cl(X,x^\pm)$, we obtain
\[
3\iota^+ = y^+(1+2a),
\qquad
1+2a = z^+ \iota^+,
\]
\[
3\iota^- = -y^-(2b+2c+1),
\qquad
-(2b+2c+1) = z^- \iota^-
\]
with positive integers $y^\pm$ and $z^\pm$.
We conclude $y^+z^+=3$ and $y^-z^-=3$.
This leaves us with the following four cases:

\medskip

\noindent
\emph{Case 1.1: $y^+=3$, $y^-=3$}.
Then we have $\iota^+ = 1+2a$ and
$\iota^- = -2b-2c-1$.
Solving for $a$ in the first equation,
for $b$ in the second one and substituting
gives
{\small
\[
u^+
 = 
\left[
\iota^+, \,  \frac{\iota^+-3}{2}, \,  3
\right],
\qquad
u^-
 = 
\left[
\iota^- + 3c, \,  \frac{\iota^-+3}{2}, \,  -3 
\right].
\]
}

\noindent
We conclude that $\iota^+$ as well as  $\iota^-$
is odd and none of them is divisible by three.
Substituting also in $P$ and the conditions on
its entries leads to setting $S_{11}(2,\iota)$.

\medskip

\noindent
\emph{Case 1.2: $y^+=3$, $y^-=1$}.
Then we have $\iota^+ = 1+2a$ and
$3\iota^- = -2b-2c-1$.
Solving for $a$ in the first equation,
for $b$ in the second one and substituting
gives
{\small
\[
u^+
 = 
\left[
\iota^+, \,  \frac{\iota^+-3}{2}, \,  3
\right],
\qquad
u^-
 = 
\left[
\iota^- + c, \,  \frac{\iota^-+1}{2}, \,  -1 
\right].
\]
}

\noindent
We conclude that $\iota^+$ as well as  $\iota^-$
is odd and $\iota^+$ is not divisible by three.
Substituting also in $P$ and the conditions on
its entries leads to setting $S_{12}(2,\iota)$.

\medskip

\noindent
\emph{Case 2.1: $y^+=1$, $y^-=3$}.
Then we have $3\iota^+ = 1+2a$ and
$\iota^- = -2b-2c-1$.
Solving for $a$ in the first equation,
for $b$ in the second one and substituting
gives
{\small
\[
u^+
 = 
\left[
\iota^+, \,  \frac{\iota^+-1}{2} \,  1,
\right],
\qquad
u^-
 = 
\left[
\iota^- + 3c, \,  \frac{\iota^-+3}{2}, \,  -3 
\right].
\]
}

\noindent
We conclude that $\iota^+$ as well as  $\iota^-$
is odd and $\iota^-$ is not divisible by three.
Substituting also in $P$ and the conditions on
its entries leads to setting $S_{21}(2,\iota)$.

\medskip

\noindent
\emph{Case 2.2: $y^+=1$, $y^-=1$}.
Then we have $3\iota^+ = 1+2a$ and
$3\iota^- = -2b-2c-1$.
Solving for $a$ in the first equation,
for $b$ in the second one and substituting
gives
{\small
\[
u^+
 = 
\left[
\iota^+, \,  \frac{\iota^+-1}{2}, \,  1
\right],
\qquad
u^-
 = 
\left[
\iota^- + c, \,  \frac{\iota^-+1}{2}, \,  -1
\right].
\]
}

\noindent
We conclude that $\iota^+$ as well as  $\iota^-$
is odd.
Substituting also in $P$ and the conditions on
its entries leads to setting $S_{22}(2,\iota)$.

\medskip

We showed that every full intrinsic quadric
surface of Picard number two is isomorphic
to some $X(P_\eta)$ with $P_\eta$ as in the assertion.
Moreover, $x_0$ and $x_1$ are of
local Gorenstein index one,
see~\cite{HaHaSp}*{Prop.~8.8~(iii)},
we obtain that 
$X(P_\eta)$ has Gorenstein index
$\iota = \lcm(\iota^+,\iota^-)$.
Conversely, one directly checks that every
matrix~$P$ from the assertion defines a
full intrinsic quadric
surface of Picard number two and
Gorenstein index
$\iota = \lcm(\iota^+,\iota^-)$.

Finally, we want to see that the matrices $P_\eta$ 
listed in the assertion define pairwise non-isomorphic $X(P_\eta)$.
Due to Proposition~\ref{prop:fiqs2matrix-rho2},
this means to show that the sets
$S_{ij}(2,\iota)$ are pairwise disjoint.
With the aid of Proposition~\ref{prop:fiqs-rho2-props-1},
we compare the local Gorenstein indices $\iota^\pm$
and the local class group orders $\cl(X,x^\pm)$:
\renewcommand{\arraystretch}{1.5}
\begin{center}
\begin{tabular}{c|c|c|c|c}
&
$S_{11}(2,\iota)$
&
$S_{12}(2,\iota)$
&
$S_{21}(2,\iota)$
&
$S_{22}(2,\iota)$
\\
\hline
$(\iota^+,\cl(X,x^+))$
&
$(\iota^+,\iota^+)$
&
$(\iota^+,\iota^+)$
&
$(\iota^+,3\iota^+)$
&
$(\iota^+,3\iota^+)$
\\
\hline
$(\iota^-,\cl(X,x^-))$
&
$(\iota^-,\iota^-)$
&
$(\iota^-,3\iota^-)$
&
$(\iota^-,\iota^-)$
&
$(\iota^-,3\iota^-)$
\end{tabular}
\end{center}

\noindent
The listed pairs are invariants of the surface
$X(P_\eta)$ up to switching $x^+$ and $x^-$.
Thus, we see that $S_{11}(2,\iota)$ as well as
$S_{22}(2,\iota)$ has trivial intersection
with any other $S_{ij}(2,\iota)$.
For $S_{12}(2,\iota)$ observe $\iota^+ < 3 \iota^-$,
as we have $3 \nmid \iota^+$.
Similarly, $3\iota^+ < \iota^-$ holds for
$S_{21}(3,\iota)$.
Thus, in both cases, $\cl(X,x^+)$ is the strictly
smallest of $\cl(X,x^\pm)$.
Consequently, $S_{12}(2,\iota)$ and $S_{21}(2,\iota)$
intersect trivially.
\end{proof}

\section{Picard number three}
\label{sec:rho-3}

The main result of this section,
Theorem~\ref{thm:fiqs2matrix-rho3},
provides the description of all
full intrinsic quadric surfaces
of Picard number three in terms of
the local Gorenstein indices of
two of their possibly singular
and the local class group orders
of two further possibly singular
points.

\begin{construction}[Full intrinsic quadric
surfaces $X$ of Picard number three
as $\KK^*$-surfaces]
\label{constr:matrix2fiqs-rho3}
Consider an integral matrix of the
form
\[
P
 := 
[v_1,v_2,v_3,v_4,v_5,v_6]
 :=  
\left[
\begin{array}{cccccc}
-1 & -1 & 1 & 1 & 0 & 0
\\
-1 & -1 & 0 & 0 & 1 & 1
\\
a & b  & 0 & c & 0 & d
\end{array}
\right],
\quad
\begin{array}{l}
a > b, \ 0 > c, \ 0 > d,
\\
a - b \ge -c \ge -d,
\\
b+c+d < 0 < a,
\\
a \le -b-c-d.
\end{array}
\]
Let $Z$ be the toric variety arising from the fan 
$\Sigma$ in $\ZZ^3$ with generator matrix $P$
and the maximal cones 
\[
\sigma^+ := \cone(v_1,v_3,v_5),
\qquad
\sigma^- := \cone(v_2,v_4,v_6),
\]
\[
\tau_0 := \cone(v_1,v_2),
\quad
\tau_1 := \cone(v_3,v_4),
\quad
\tau_2 := \cone(v_5,v_6).
\]
Denote by $U_1,U_2,U_3$ the coordinate
functions on the standard 3-torus $\TT^3 \subseteq Z$.
Then we obtain a normal, non-toric, rational, projective
surface
\[
X \ := \ X(P) \ := \ \overline{V(h)} \ \subseteq \ Z,
\qquad
h \ := \ 1 + U_1 + U_2 \ \in \mathcal{O}(\TT^3)
\]
with a $\KK^*$-action, given on $V(h) \subseteq \TT^3$ by
$t \cdot x = (x_1,x_2,tx_3)$.
The divisor class group of~$X$ equals that
of $Z$ and is given by 
\[
\Cl(X) \ \cong \ \Cl(Z)
\ \cong \ 
K
\ := \
\ZZ^3 \times \ZZ \, / \gcd(a,b,c,d) \ZZ.
\]
Moreover, denoting by $Q \colon\ZZ^6 \to K$ the projection,
we obtain the following description of the Cox ring of $X$
as a graded algebra:
\[
\mathcal{R}(X) 
\ \cong \   
\KK[T_1, \ldots, T_6] / \bangle{T_1T_2+T_3T_4+T_5T_6},
\qquad
\deg(T_i) = Q(e_i) = [D_i],
\]
where $D_i \subseteq X$ is the prime divisor on $X$ obtained by
intersecting $X$ with the toric prime divisor of $Z$
given by the ray through $v_i$ and $[D_i] \in \Cl(X)$ denotes
its class.
\end{construction}

\begin{proof}
By definition, the columns $v_1,\ldots,v_6$ of $P$ are pairwise
different primitive integral vectors. Moreover, $v_1,\ldots,v_6$
generate $\QQ^3$ as a convex cone, as we have
\[
v_1 + v_3 + v_5 =  [0,0,a], \ a > 0,
\qquad
v_2 + v_4 + v_6 =  [0,0,b+c+d], \ b+c+d < 0.
\]
Consequently, $P$ is a defining matrix of a rational projective
$\KK^*$-surface in the sense of~\cite{ArDeHaLa}*{Constr.~3.4.3.1}.
The statements follow from those
of~\cite{ArDeHaLa}*{Sec.~3.4.3}.
\end{proof}

\begin{proposition}
\label{prop:fiqs-rho3-props-1}
Let $X = X(P)$ arise from
Construction~\ref{constr:matrix2fiqs-rho3}.
The fixed points of the $\KK^*$-action on $X$ are
given in Cox coordinates by
\[
x^+ \ := \ [0,1,0,1,0,1],
\qquad
x^- \ := \ [1,0,1,0,1,0],
\]
\[
x_0 \ := \ [0,0,1,1,1,-1],
\quad
x_1 \ := \ [1,1,0,0,1,-1],
\quad
x_2 \ := \ [1,1,1,-1,0,0].
\]
Moreover, the orders of the local class groups of
the fixed points of the $\KK^*$-action
are given by 
\[
\cl(X,x^+) = a,
\qquad
\cl(X,x^-) = -b-c-d,
\]
\[
\cl(X,x_0) = a-b,
\quad
\cl(X,x_1) = -c,
\quad
\cl(X,x_2) = -d.
\]
Finally, the ordered tuples $(a, \, -b-c-d)$
and $(a-b, \, -c, \, -d)$ are isomorphy invariants
of the algebraic surface $X$.
\end{proposition}

\begin{proof}
As in the previous section,
the references from the proof of
Proposition~\ref{prop:fiqs-rho1-props-1}
deliver the description of the fixed points
and show that the local class
group orders of
$x^+$, $x^-$, $x_0$, $x_1$ and $x_2$ are
\[
\det[v_1,v_3,v_5],
\qquad\qquad
\det[v_2,v_4,v_6],
\]
\[
\det
\left[\renewcommand{\arraystretch}{.7} \arraycolsep=.7\arraycolsep 
  \begin{array}{rr}
  -1 &  -1 \\ a &  b 
  \end{array}
\right],
\qquad
-\det
\left[\renewcommand{\arraystretch}{.7} \arraycolsep=.7\arraycolsep 
  \begin{array}{rr}
  1 &  1 \\ 0 &  c
  \end{array}
\right],
\qquad
-\det
\left[\renewcommand{\arraystretch}{.7} \arraycolsep=.7\arraycolsep 
  \begin{array}{rr}
  1 &  1 \\ 0 &  d
  \end{array}
\right].
\]
Simlarly as in the corresponding earlier proofs,
$x^+$, $x^-$ are the only fixed points
lying in the closure of infinitely many orbits
and each of $x_0$, $x_1$, $x_2$ lies in the
closure of precisely two non-trivial orbits.
Thus, the sets $\{\cl(X,x^+),\cl(X,x^-)\}$ and
$\{\cl(X,x_0),\cl(X,x_1),\cl(X,x_2)\}$ are
invariants of the $\KK^*$-surface $X$.
Again, the assertion follows from the fact
that on a non-toric, rational, projective,
surface any two $\KK^*$-actions
are conjugate in the automorphism group.
\end{proof}

\begin{proposition}
\label{prop:fiqs2matrix-rho3}
Every full intrinsic quadric surface $X$ 
of Picard number three is isomorphic to
an $X(P)$ for precisely one matrix $P$
from Construction~\ref{constr:matrix2fiqs-rho3}.
\end{proposition}

\begin{proof}
Applying once more Theorem~\ref{thm:fiqs-cr}
and~\cite{HaHiWr}*{Ex.~7.1} yields
that the defining matrix~$P$
is of the format $3 \times 6$ and the
first two rows look as wanted:
\[
P
\ = \ 
\left[
\begin{array}{cccccc}
-1 & -1 & 1 & 1 & 0 & 0
\\
-1 & -1 & 0 & 0 & 1 & 1
\\
d_1 & d_2 & d_3 & d_4 & d_5 & d_6
\end{array}
\right].
\]
Again, suitable admissible operations
bring us to the setting of 
Construction~\ref{constr:matrix2fiqs-rho3}. 
First, adding suitable multiples of the first two rows
to the last one, we achieve
\[
P
\ = \ 
\left[
\begin{array}{cccccc}
-1 & -1 & 1 & 1 & 0 & 0
\\
-1 & -1 & 0 & 0 & 1 & 1
\\
a & b & 0 & c & 0 & d
\end{array}
\right].
\]
Second, swapping columns inside the pairs $(v_1,v_2)$,
$(v_3,v_4)$ and $(v_5,v_6)$ and re-arranging via the
first step, we achieve that $P$ is slope-ordered, meaning
\[
a > b, \quad 0 > c, \quad 0 > d.
\]
Third, suitable swapping the columns blocks
$[v_1,v_2]$, $[v_3,v_4]$ and $[v_5,v_6]$
and re-adjusting the entries, we can ensure
\[
a - b  \ge - c \ge -d.
\]
As for any defining matrix of a rational $\KK^*$-surface
with two elliptic fixed points, slope orderedness implies 
\[
a =: m^+ > 0,
\qquad\qquad
b+c+d =: m^- < 0.
\]  
Multiplying the last row by $-1$ turns $m^{\pm}$
into $m^{\mp}$. Doing so, if necessary, and re-arranging
via the first two steps yields
\[
a
\ \le \
-b-c-d.
\]

We show that $X(P) \cong X(P')$ with matrices $P$
and $P'$ as in Construction~\ref{constr:matrix2fiqs-rho3}
implies $P = P'$.
Proposition~\ref{prop:fiqs-rho3-props-1}
yields equality of the ordered tuples 
\[
(a, \, -b-c-d)
\ = \
(a', \, -b'-c'-d'),
\qquad
(a-b, \, -c, \, -d)
\ = \
(a'-b', \, -c', \, -d')
\]
built from the entries of the third row of
$P$ and $P'$ respectively.
From this we directly derive $P = P'$.
\end{proof}

\begin{theorem}
\label{thm:fiqs2matrix-rho3}
For any $\iota \in \ZZ_{\ge 1}$, consider
the set~$M_\iota$ of 4-tuples
$\eta = (\iota^+,\iota^-,c,d)$, where
$\iota^+,\iota^- \in \ZZ_{\ge 1}$ with
$\lcm(\iota^+,\iota^-) = \iota$
and $c,d \in \ZZ_{\le -1}$.
Define subsets
\[
\begin{array}{lcl}
S_{11}(3,\iota)
& := & 
\{
\eta \in M_\iota;
\ 2 \nmid \iota^+,\iota^-,
\ \iota^+ \le \iota^-,
\ - \iota^+ - \iota^- \le 2c + d,
\ c \le d  \le -1
\},
\\[7pt]       
S_{12}(3,\iota)
& := &
\{
\eta \in M_\iota;
\ 2 \nmid \iota^+,
\ \iota^+ \le 2\iota^-,
\ - \iota^+ - 2\iota^- \le 2c + d,
\ c \le d  \le -1
\},
\\[7pt]       
S_{21}(3,\iota)
& := &
\{
\eta \in M_\iota;
\ 2 \nmid \iota^-,
\ 2\iota^+ \le \iota^-,
\ - 2\iota^+ - \iota^- \le 2c + d,
\ c \le d  \le -1
\}, 
\\[7pt]       
S_{22}(3,\iota)
& := & 
\{
\eta \in M_\iota;
\ \iota^+ \le \iota^-,
\ - 2\iota^+ - 2\iota^- \le 2c + d,
\ c \le d  \le -1
\}.
\end{array}
\]

\noindent
Then each set $S_{ij}(3,\iota)$ provides us with a series
of defining matrices $P_\eta$ of full intrinsic quadric
surfaces:

{\small
\[
\arraycolsep=.7\arraycolsep
\begin{array}{lcl}
\eta = (\iota^+,\iota^-,c,d) \in S_{11}(3,\iota)\colon
& &
\eta = (\iota^+,\iota^-,c,d) \in S_{12}(3,\iota)\colon
\\[5pt]
P_\eta =
\left[
\begin{array}{cccccc}
-1 & -1 & 1 & 1 & 0 & 0
\\
-1 & -1 & 0 & 0 & 1 & 1
\\
\iota^+ & -\iota^- -c-d & 0 & c & 0 & d 
\end{array}
\right],
&  &
P_\eta  =
\left[
\begin{array}{cccccc}
-1 & -1 & 1 & 1 & 0 & 0
\\
-1 & -1 & 0 & 0 & 1 & 1
\\
\iota^+ & -2\iota^- -c-d & 0 & c & 0 & d 
\end{array}
\right],
\end{array}
\]

%%%%
%%%%
%%%% \\[30pt]
%%%%
%%%%  

\[
\arraycolsep=.7\arraycolsep
\begin{array}{lcl}
\eta = (\iota^+,\iota^-,c,d) \in S_{21}(3,\iota)\colon
& &
\eta = (\iota^+,\iota^-,c,d) \in S_{22}(3,\iota)\colon
\\[5pt]
P_\eta  =
\left[
\begin{array}{cccccc}
-1 & -1 & 1 & 1 & 0 & 0
\\
-1 & -1 & 0 & 0 & 0 & 0
\\
2\iota^+ & -\iota^- -c-d & 0 & c & 0 & d 
\end{array}
\right],
&  &
P_\eta  =
\left[
\begin{array}{cccccc}
-1 & -1 & 1 & 1 & 0 & 0
\\
-1 & -1 & 0 & 0 & 1 & 1
\\
2\iota^+ & -2\iota^- -c -d & 0 & c & 0 & d 
\end{array}
\right].
\end{array}           
\]
}

\noindent
Each $X(P_\eta)$ is of Picard number three,
Gorenstein index $\iota = \lcm(\iota^+,\iota^-)$
and $\iota^+$,~$\iota^-$, resp. $-c$, $-d$ are the local
Gorenstein indices resp. local class group orders
of 
\[
x^+ \ = \ [0,1,0,1,0,1], \qquad
x^- \ = \ [1,0,1,0,1,0],
\]
\[
x_1 \ = \ [1,1,0,0,1,-1], \qquad 
x_2 \ = \ [1,1,1,-1,0,0].
\]
Finally, every full intrinsic quadric surface of Picard number three
and Gorenstein index~$\iota$ is isomorphic to $X(P_\eta)$ for
precisely one $P_\eta$ from the above list.
\end{theorem}

\goodbreak

\begin{proof}
Let $X$ be a full intrinsic quadric surface
of Picard number three.  
Then Construction~\ref{constr:matrix2fiqs-rho3}
and Proposition~\ref{prop:fiqs2matrix-rho3}
allow us to assume $X = X(P)$ with
\[
P
\ = \ 
\left[
\begin{array}{cccccc}
-1 & -1 & 1 & 1 & 0 & 0
\\
-1 & -1 & 0 & 0 & 1 & 1
\\
a & b  & 0 & c & 0 & d
\end{array}
\right],
\qquad
\begin{array}{l}
a > b, \ 0 > c, \ 0 > d,
\\
a - b \ge -c \ge -d,
\\
b+c+d < 0 < a,
\\
a \le -b-c-d.
\end{array}
\]
Consider anticanonical divisor
$-\mathcal{K} = D_3+D_4+D_5+D_6$
on $X(P)$.
The linear forms $u^\pm$ representing
the $\iota^\pm$-fold of $-\mathcal{K}$
near $x^\pm$ are given as

{\small
\[
u^+
 = 
\left[
\iota^+, \,  \iota^+, \,  \frac{2\iota^+}{a}
\right],
\qquad
u^-
 = 
\left[
\frac{(b-c+d)\iota^-}{b+c+d}, \,  \frac{(b+c-d)\iota^-}{b+c+d}, \,  -\frac{2\iota^-}{b+c+d} 
\right].
\]
}

\noindent
By the definition of the local Gorenstein index, these 
are primitive integral vectors.
Together with the fact that $\iota^\pm$ divides
$\cl(X,x^\pm)$, we obtain
\[
2\iota^+ = y^+a,
\qquad
a = z^+ \iota^+,
\]
\[
2\iota^- = -y^-(b+c+d),
\qquad
-(b+c+d) = z^- \iota^-
\]
with positive integers $y^\pm$ and $z^\pm$.
We conclude $y^+z^+=2$ and $y^-z^-=2$.
The possible constellations of $(y^+,y^-)$
yield the following four cases:

\medskip

\noindent
\emph{Case 1.1: $a=\iota^+$, $b = -\iota^--c-d$}.
Inserting this, we see that $P$
arises from $S_{11}(3,\iota)$ and its entries
satisfy the required estimates.
Moreover, $u^\pm$ become

{\small
\[
u^+
 = 
\left[
\iota^+, \,  \iota^+, \,  2 
\right],
\qquad\qquad
u^-
 = 
\left[
2c+\iota^-, \,  2d+\iota^-, \,  -2 
\right].
\]
}

\noindent
As these are integral primitive vectors,
we see that $\iota^+$ as well as $\iota^-$
are odd and that $\iota^\pm$ are indeed
the local Gorenstein indices of $x^\pm$.

\medskip

\noindent
\emph{Case 1.2: $a=\iota^+$, $b = -2\iota^--c-d$}.
As in the previous subcase, inserting shows that
$P$ stems from $S_{12}(3,\iota)$.
Note that this time  we have

{\small
\[
u^+
 = 
\left[
\iota^+, \,  \iota^+, \,  2 
\right],
\qquad\qquad
u^-
 = 
\left[
c+\iota^-, \,  d+\iota^-, \,  -1
\right].
\]
}

\noindent
Thus, $\iota^+$ is odd and we have no divisibility
condition on $\iota^-$.
As before, we obtain that $\iota^\pm$ are indeed
the local Gorenstein indices of $x^\pm$.

\medskip

\noindent
\emph{Case 2.1: $a=2\iota^+$, $b = -\iota^--c-d$}.
Inserting shows that $P$ is given by 
$S_{21}(3,\iota)$ and its entries
satisfy the required estimates.
Moreover, we have 

{\small
\[
u^+
 = 
\left[
\iota^+, \,  \iota^+, \,  1
\right],
\qquad\qquad
u^-
 = 
\left[
2c+\iota^-, \,  2d+2\iota^++\iota^-, \,  -2 
\right].
\]
}

\noindent
These must be integral primitive vectors.
Consequently,  $\iota^-$ is odd and we
obtain that $\iota^\pm$ are indeed
the local Gorenstein indices of $x^\pm$.

\medskip

\noindent
\emph{Case 2.2: $a=2\iota^+$, $b = -2\iota^--c-d$}.
Inserting shows that the matrix $P$ arises from
$S_{22}(3,\iota)$.
Moreover, the linear forms $u^\pm$ are given by

{\small
\[
u^+
 = 
\left[
\iota^+, \,  \iota^+, \,  1
\right],
\qquad\qquad
u^-
 = 
\left[
c+\iota^-, \,  d+\iota^++\iota^-, \,  -1
\right].
\]
}

\noindent
Thus, there are no divisibility
conditions on $\iota^\pm$ and we
see that $\iota^\pm$ are indeed
the local Gorenstein indices of $x^\pm$.

\medskip

We showed that every full intrinsic quadric
surface of Picard number three is isomorphic
to some $X(P)$ with $P$ as in the assertion.
Moreover, as $x_0,x_1,x_2 \in X(P)$ are all of
local Gorenstein index one,
see~\cite{HaHaSp}*{Prop.~8.9~(iii)},
we obtain that 
$X(P)$ has Gorenstein index
$\iota = \lcm(\iota^+,\iota^-)$.
Conversely, one directly checks that every
matrix~$P$ from the assertion defines a
full intrinsic quadric
surface of Picard number three and
Gorenstein index
$\iota = \lcm(\iota^+,\iota^-)$.

Finally, we want to see that the matrices $P$ 
listed in the assertion define pairwise non-isomorphic $X(P)$.
According to Proposition~\ref{prop:fiqs2matrix-rho3},
this amounts to showing that the sets
$S_{ij}(3,\iota)$ are pairwise disjoint.
We use Proposition~\ref{prop:fiqs-rho3-props-1}
to compare the local Gorenstein indices $\iota^\pm$
and the local class group orders $\cl(X,x^\pm)$:
\renewcommand{\arraystretch}{1.5}
\begin{center}
\begin{tabular}{c|c|c|c|c}
&
$S_{11}(3,\iota)$
&
$S_{12}(3,\iota)$
&
$S_{21}(3,\iota)$
&
$S_{22}(3,\iota)$
\\
\hline
$(\iota^+,\cl(X,x^+))$
&
$(\iota^+,\iota^+)$
&
$(\iota^+,\iota^+)$
&
$(\iota^+,2\iota^+)$
&
$(\iota^+,2\iota^+)$
\\
\hline
$(\iota^-,\cl(X,x^-))$
&
$(\iota^-,\iota^-)$
&
$(\iota^-,2\iota^-)$
&
$(\iota^-,\iota^-)$
&
$(\iota^-,2\iota^-)$
\end{tabular}
\end{center}

\noindent
The listed pairs are invariants of the surface
up to switching $x^+$ and $x^-$.
Thus, we see that $S_{11}(3,\iota)$ as well as
$S_{22}(3,\iota)$ has trivial intersection
with any other $S_{ij}(3,\iota)$.
For $S_{12}(3,\iota)$ observe $\iota^+ < 2 \iota^-$
as $\iota^+$ is odd.
Similarly, for $S_{21}(3,\iota)$, we have
$2\iota^+ < \iota^-$.
Thus, in both cases, $\cl(X,x^+)$ is the strictly
smallest of $\cl(X,x^\pm)$.
It follows that  $S_{12}(3,\iota)$ and $S_{21}(3,\iota)$
intersect trivially.
\end{proof}

\section{Geometry of full intrinsic quadric surfaces}
\label{sec:geom-app}

We present direct applications of
Theorems~\ref{thm:fiqs2matrix-rho1},~\ref{thm:fiqs2matrix-rho2}
and~\ref{thm:fiqs2matrix-rho3},
exploring the geometry of full intrinsic quadric surfaces.
In Corollary~\ref{cor:resolvesing} we determine the
weighted resolution graphs for the canonical resolution
of singularities.
Moreover, Corollaries~\ref{cor:anticandeg}, ~\ref{cor:maxlc}
and~\ref{cor:picind} give explicit upper and lower bounds
on the degree, the log canonicity
and the Picard index in terms of the Gorenstein index.
Finally, Corollary~\ref{cor:ke-metrics} characterizes
the existence of K\"ahler-Einstein metrics
in terms of the Gorenstein index.

First, recall that a \emph{del Pezzo surface} is a
normal projective surface admitting an ample anticanonical
divisor $-\mathcal{K}_X$.
Moreover, a del Pezzo surface is \emph{log terminal}
if all the exceptional divisors of its minimal resolution
of singularities have discrepancies strictly bigger
than $-1$.

\goodbreak

\begin{proposition}
\label{prop:fiqs2fano}
Every full intrinsic quadric surface $X$ is a log del Pezzo 
surface.
\end{proposition}

\begin{proof}
We may assume $X = X(P)$.
Then log terminality is direct a conseqence
of~\cite{HaHaSp}*{Cor.~8.12}.
According to the possible values of the Picard
number $\rho = \rho(X)$, the degree
$\mu \in \Cl(X)$ of the defining
quadric of $X$ is given as
\[
\mu \ = \ 
\begin{cases}
w_1+w_2 = 2w_3 = 2w_4,           & \rho = 1,
\\
w_1+w_2 = w_3+w_4 = 2w_5,        & \rho = 2,
\\
w_1+w_2 = w_3+w_4 = w_5+w_6 ,    & \rho = 3,
\end{cases}
\]
where $w_i = \deg(T_i) \in \Cl(X)$.
Due to~\cite{ArDeHaLa}*{Prop.~3.3.3.2},
the anticanonical class of $X$ equals
$w_1+\ldots+w_{\rho+3}-\mu$ and thus is
a positive multiple of $\mu$.
From~\cite{ArDeHaLa}*{Prop.~3.3.2.9}
we infer that the cone of movable
divisor classes of $X$ is given by
\[
\Mov(X)
\ = \
\bigcap_{i=1}^{\varrho+3} \tau_i,
\qquad
\tau_i
\ := \ 
\cone(w_j; \ j \ne i)
\ \subseteq \
\Cl_\QQ(X).
\]
Observe that $\mu$ is an interior point 
of each $\tau_i$.
As all involved cones are of full dimension,
we obtain that $\mu$ is an interior point
of $\Mov(X)$.
Thus,~\cite{ArDeHaLa}*{Prop.~3.3.2.9, Thm.~4.3.3.5}
show that $\mu$, and hence the anticanonical class of
$X$, is ample.
\end{proof}

The \emph{(anticanonical) degree} of a del Pezzo surface
is the self intersection number $\mathcal{K}_X^2$ of an
anticanonical divisor of $X$.
For the full intrinsic quadric surfaces,
we obtain the following relations between the degree
and the local Gorenstein indices.

\begin{corollary}
\label{cor:anticandeg}
Consider a full intrinsic quadric surface
$X = X(P_\eta)$ with $P_\eta$ as in
Theorem~\ref{thm:fiqs2matrix-rho1},~\ref{thm:fiqs2matrix-rho2}
or~\ref{thm:fiqs2matrix-rho3}.
Then the degree $\mathcal{K}_X^2$ of $X$ is given as 
\[
\begin{array}{lcllcll}
\rho=1: 
&&
\mathcal{K}_X^2 = \frac{1}{\iota^+} + \frac{1}{\iota^-},
&
\eta \in S_{11}(1,\iota),
&&
\mathcal{K}_X^2 = \frac{1}{\iota^+} + \frac{2}{\iota^-},
&
\eta \in S_{12}(1,\iota),
\\[5pt]
&&
\mathcal{K}_X^2 = \frac{2}{\iota^+} + \frac{1}{\iota^-},
&
\eta \in S_{21}(1,\iota),
&&
\mathcal{K}_X^2 = \frac{2}{\iota^+} + \frac{2}{\iota^-},
&
\eta \in S_{22}(1,\iota),
\\[10pt]
\rho=2: 
&&
\mathcal{K}_X^2 = \frac{9}{2\iota^+} + \frac{9}{2\iota^-},
&
\eta \in S_{11}(2,\iota),
&&
\mathcal{K}_X^2 = \frac{9}{2\iota^+} + \frac{3}{2\iota^-},
&
\eta \in S_{12}(2,\iota),
\\[5pt]
&&
\mathcal{K}_X^2 = \frac{3}{2\iota^+} + \frac{9}{2\iota^-},
&
\eta \in S_{21}(2,\iota),
&&
\mathcal{K}_X^2 = \frac{3}{2\iota^+} + \frac{3}{2\iota^-},
&
\eta \in S_{22}(2,\iota),
\\[10pt]
\rho=3: 
&&
\mathcal{K}_X^2 = \frac{4}{\iota^+} + \frac{4}{\iota^-},
&
\eta \in S_{11}(3,\iota),
&&
\mathcal{K}_X^2 = \frac{4}{\iota^+} + \frac{2}{\iota^-},
&
\eta \in S_{12}(3,\iota),
\\[5pt]
&&
\mathcal{K}_X^2 = \frac{2}{\iota^+} + \frac{4}{\iota^-},
&
\eta \in S_{21}(3,\iota),
&&
\mathcal{K}_X^2 = \frac{2}{\iota^+} + \frac{2}{\iota^-},
&
\eta \in S_{22}(3,\iota).
\\
\end{array}
\]
Here, $\rho$ is the Picard number, $\iota$ the Gorenstein index
of $X$ and $\iota^\pm$ the local Gorenstein index of  $x^\pm \in X$.
Moreover, we obtain the following upper and lower
bounds:
\[
\begin{array}{ll}
\rho= 1: & \frac{2}{\iota} \le \mathcal{K}_X^2 \le  1 + \frac{4}{\iota},
\\[5pt]
\rho= 2: & \frac{3}{\iota} \le \mathcal{K}_X^2 \le  \frac{9}{2} + \frac{9}{2\iota},
\\[5pt] 
\rho= 3: & \frac{4}{\iota} \le \mathcal{K}_X^2 \le  4 + \frac{4}{\iota}.
\end{array}
\]
\end{corollary}

\begin{proof}
First assume $X = X(P)$ with $P$ from
Construction~\ref{constr:matrix2fiqs-rho1},~\ref{constr:matrix2fiqs-rho2}
or~\ref{constr:matrix2fiqs-rho3}.
Taking any representative of the anticanonical class
as listed in the proof of the preceding proposition,
we use the intersection numbers provided by~\cite{HaHaSp}*{Summary~7.7},
and compute according to the cases $\rho=1,2,3$:
\[
\mathcal{K}_X^2 = \frac{1}{a+1} - \frac{1}{b+1},
\qquad
\mathcal{K}_X^2 = \frac{9}{4a+2} - \frac{9}{2+4b+4c},
\qquad
\mathcal{K}_X^2 = \frac{4}{a} - \frac{4}{b+c+d}.
\]
Inserting the values for $a,b,c,d$ from 
Theorems~\ref{thm:fiqs2matrix-rho1},~\ref{thm:fiqs2matrix-rho2}
and~\ref{thm:fiqs2matrix-rho3} accordingly,
we obtain the desired presentations of the anticanonical
self intersection number.
The estimates are then directly verified.
\end{proof}

\def\ellgraph#1{
\begin{tikzpicture}[scale=0.6]
\sffamily
\draw[thick] (0,0) circle (2pt);
\node[] at (0,.5) {#1};
\draw[thick] (-1.5,0) circle (2pt);
\node[] at (-1.5,.5) {$\scriptscriptstyle -2$};
\draw[thick] (1.5,0) circle (2pt);
\node[] at (1.5,.5)  {$\scriptscriptstyle -2$};
\draw[] (-1.4,0) edge (-.1,0);
\draw[] (.1,0) edge (1.4,0);
\end{tikzpicture}
}

\def\eellgraph#1{
\begin{tikzpicture}[scale=0.6]
\sffamily
\draw[thick] (0,0) circle (2pt);
\node[] at (0,.5) {#1};
\draw[thick] (-1.5,0) circle (2pt);
\node[] at (-1.5,.5) {$\scriptscriptstyle -2$};
\draw[] (-1.4,0) edge (-.1,0);
\end{tikzpicture}
}

\def\eeellgraph#1{
\begin{tikzpicture}[scale=0.6]
\sffamily
\draw[thick] (0,0) circle (2pt);
\node[] at (0,.5) {#1};
\end{tikzpicture}
}

\def\hypgraph{
\begin{tikzpicture}[scale=0.6]
\sffamily
\draw[thick] (-1.5,0) circle (2pt);
\node[] at (-1.5,.5) {$\scriptscriptstyle -2$};
\draw[thick] (-.5,0) circle (2pt);
\node[] at (-.5,.5) {$\scriptscriptstyle -2$};
\draw[thick] (1.5,0) circle (2pt);
\node[] at (1.5,.5)  {$\scriptscriptstyle -2$};
\draw[] (-1.4,0) edge (-.6,0);
\draw (-.4,0) edge (-.1,0);
\draw[dotted] (0,0) edge (1,0);
\draw (1.1,0) edge (1.4,0);
\end{tikzpicture}
}

\def\sij#1{
\begin{tikzpicture}[scale=0.6]
\sffamily
\node[] at (0,0) {#1};
\end{tikzpicture}
}

We turn to the singularities of full intrinsic
quadrics and consider the canonical resolution
of singularities in the sense of ~\cite{OrWa};
see also~\cite{ArDeHaLa}*{Sec.~5.4.3}.

\begin{corollary}
\label{cor:resolvesing}
Consider a full intrinsic quadric surface
$X = X(P_\eta)$ with $P_\eta$ as in
Theorem~\ref{thm:fiqs2matrix-rho1},~\ref{thm:fiqs2matrix-rho2}
or~\ref{thm:fiqs2matrix-rho3}
and its canonical resolution of singularities.
Then the possible singularities $x^+,x^-,x_0,x_1,x_2 \in X$
have the following resolution graphs:

{\small
\[
\begin{array}{lccc}    
\rho=1:
&
\scriptstyle x^+, \ 3
&
\scriptstyle x^-, \ 3
&
\scriptstyle x_0, \ \iota^++\iota^--1
\\[10pt]
\sij{$\scriptstyle S_{11}(1,\iota)$}
&
\ellgraph{$\scriptscriptstyle -1-\iota^+$}
&
\ellgraph{$\scriptscriptstyle -1-\iota^-$}
&
\hypgraph
\\[5pt]
\sij{$\scriptstyle S_{12}(1,\iota)$}
&
\ellgraph{$\scriptscriptstyle -1-\iota^+$}
&
\ellgraph{$\scriptscriptstyle -1-\frac{\iota^-}{2}$}
&
\hypgraph
\\[5pt]
\sij{$\scriptstyle S_{21}(1,\iota)$}
&
\ellgraph{$\scriptscriptstyle -1-\frac{\iota^+}{2}$}
&
\ellgraph{$\scriptscriptstyle -1-\iota^-$}
&
\hypgraph
\\[5pt]
\sij{$\scriptstyle S_{22}(1,\iota)$}
&
\ellgraph{$\scriptscriptstyle -1-\frac{\iota^+}{2}$}
&
\ellgraph{$\scriptscriptstyle -1-\frac{\iota^-}{2}$}
&
\hypgraph
\end{array}
\]
%%%%%%%%%%
%%%%%%%%%%
%%%%%%%%%%

\medskip

%%%%%%%%%%
%%%%%%%%%%
%%%%%%%%%%
\[
\renewcommand{\arraycolsep}{5pt}
\begin{array}{lcccc}    
\rho=2:
&
\scriptstyle x^+, \ 2
&
\scriptstyle x^-, \ 2
&
\scriptstyle x_0, \ \frac{\iota^++\iota^-}{2}+c-1+\varepsilon
&
\scriptstyle x_1, \ -1-c
\\[10pt]
\sij{$\scriptstyle S_{11}(2,\iota)$}
&
\eellgraph{$\scriptscriptstyle -\frac{1}{2}-\frac{\iota^+}{2}$}
&
\eellgraph{$\scriptscriptstyle -\frac{1}{2}-\frac{\iota^-}{2}$}
&
\hypgraph {\scriptscriptstyle \varepsilon=0}
&
\hypgraph
\\[5pt]
\sij{$\scriptstyle S_{12}(2,\iota)$}
&
\eellgraph{$\scriptscriptstyle -\frac{1}{2}-\frac{\iota^+}{2}$}
&
\eellgraph{$\scriptscriptstyle -\frac{1}{2}-\frac{3\iota^-}{2}$}
&
\hypgraph {\scriptscriptstyle \varepsilon=\iota^-}
&
\hypgraph
\\[5pt]
\sij{$\scriptstyle S_{21}(2,\iota)$}
&
\eellgraph{$\scriptscriptstyle -\frac{1}{2}-\frac{3\iota^+}{2}$}
&
\eellgraph{$\scriptscriptstyle -\frac{1}{2}-\frac{\iota^-}{2}$}
&
\hypgraph {\scriptscriptstyle \varepsilon=\iota^+}
&
\hypgraph
\\[5pt]
\sij{$\scriptstyle S_{22}(2,\iota)$}
&
\eellgraph{$\scriptscriptstyle -\frac{1}{2}-\frac{3\iota^+}{2}$}
&
\eellgraph{$\scriptscriptstyle -\frac{1}{2}-\frac{3\iota^-}{2}$}
&
\hypgraph {\scriptscriptstyle \varepsilon=\iota^++\iota^-}
&
\hypgraph
\end{array}
\]
%%%%%%%%%%
%%%%%%%%%%
%%%%%%%%%%

\medskip

%%%%%%%%%% 
%%%%%%%%%%
%%%%%%%%%%
\[
\renewcommand{\arraycolsep}{4pt}
\begin{array}{lccccc}    
\rho=3:
&
\scriptstyle x^+, \ 1
&
\scriptstyle x^-, \ 1
&
\scriptstyle x_0, \ \iota^++\iota^-+c+d-1+\varepsilon
&
\scriptstyle x_1, \ -1-c
&
\scriptstyle x_2, \ -1-d
\\[10pt]
\sij{$\scriptstyle S_{11}(3,\iota)$}
&
\eeellgraph{$\scriptscriptstyle -\iota^+$}
&
\eeellgraph{$\scriptscriptstyle -\iota^-$}
&
\hypgraph {\scriptscriptstyle \varepsilon=0}
&
\hypgraph
&
\hypgraph
\\[5pt]
\sij{$\scriptstyle S_{12}(3,\iota)$}
&
\eeellgraph{$\scriptscriptstyle -\iota^+$}
&
\eeellgraph{$\scriptscriptstyle -\iota^-$}
&
\hypgraph {\scriptscriptstyle \varepsilon=\iota^-}
&
\hypgraph
&
\hypgraph
\\[5pt]
\sij{$\scriptstyle S_{21}(3,\iota)$}
&
\eeellgraph{$\scriptscriptstyle -2\iota^+$}
&
\eeellgraph{$\scriptscriptstyle -\iota^-$}
&
\hypgraph {\scriptscriptstyle \varepsilon=\iota^+}
&
\hypgraph
&
\hypgraph
\\[5pt]
\sij{$\scriptstyle S_{22}(3,\iota)$}
&
\eeellgraph{$\scriptscriptstyle -2\iota^+$}
&
\eeellgraph{$\scriptscriptstyle -2\iota^-$}
&
\hypgraph {\scriptscriptstyle \varepsilon=\iota^++\iota^-}
&
\hypgraph
&
\hypgraph
\end{array}
\]
}

\noindent
Here, the weights of the vertices are the self intersection numbers of the
corresponding exceptional curves.
The canonical resolution is minimal unless $x^+ \in X$ is smooth,
where the latter happens if and only if
\[
\iota^+=1,
\
\eta \in S_{11}(2,\iota) \cup S_{12}(2,\iota) \cup S_{11}(3,\iota) \cup S_{12}(3,\iota).
\]
\end{corollary}

\begin{proof}
For $X = X(P)$ with $P$ as in 
Construction~\ref{constr:matrix2fiqs-rho1},~\ref{constr:matrix2fiqs-rho2}
or~\ref{constr:matrix2fiqs-rho3},
we use~\cite{ArDeHaLa}*{Sec.~5.4.3} to determine the canonical
resolution of singularities of $X$; see also~\cite{HaHaSp}*{Summary~8.2}.
Then we compute the self intersection numbers of the exceptional
divisors according to~\cite{ArDeHaLa}*{Sec.~5.4.2};
see also~\cite{HaHaSp}*{Summary~7.7}
and insert the values of $a,b,c,d$ from
Theorems~\ref{thm:fiqs2matrix-rho1},~\ref{thm:fiqs2matrix-rho2}
and~\ref{thm:fiqs2matrix-rho3}.
\end{proof}

\def\accrhoone{
\begin{tikzpicture}[scale=0.4]
\sffamily
\coordinate(ooo) at (0,0);
\coordinate(cplus) at (0,2);
\coordinate(c01) at (-1.4,1);
\coordinate(c02) at (-1.4,-1);
\coordinate(c11) at (1.4,-.7);
\coordinate(c21) at (1.8,1);
\coordinate(cminus) at (0,-2);
\path[fill, color=gray!30] (ooo) -- (cplus) -- (c21) -- (cminus) -- (ooo);
\path[fill, color=gray!10] (ooo) -- (cplus) -- (c01) -- (c02) -- (cminus) -- (ooo);
\draw[thick, color=black] (ooo) to (c01);
\draw[thick, color=black] (ooo) to (c02);
\draw[thick, color=black] (ooo) to (c21);
\draw[thick, color=black] (ooo) to (cplus);
\draw[thick, color=black] (ooo) to (cminus);
\draw[thick, color=black] (cplus) to (c01);
\draw[thick, color=black] (cplus) to (c21);
\draw[thick, color=black] (c21) to (cminus);
\draw[thick, color=black] (c01) to (c02);
\path[fill, opacity=.9, color=gray!20] (ooo) -- (cplus) -- (c11) -- (cminus) -- (ooo);
\draw[thick, color=black] (ooo) to (c11);
\draw[thick, color=black] (cplus) to (c11);
\draw[thick, color=black] (c02) to (cminus);
\draw[thick, color=black] (c11) to (cminus);
\end{tikzpicture}
}

\def\accrhotwo{
\begin{tikzpicture}[scale=0.4]
\sffamily
\coordinate(ooo) at (0,0);
\coordinate(cplus) at (0,2);
\coordinate(c01) at (-1.4,1);
\coordinate(c02) at (-1.4,-1);
\coordinate(c11) at (.9,1);
\coordinate(c12) at (.9,-1);
\coordinate(c21) at (1.6,.8);
\coordinate(cminus) at (0,-2);
\path[fill, color=gray!30] (ooo) -- (cplus) -- (c21) -- (cminus) -- (ooo);
\path[fill, color=gray!10] (ooo) -- (cplus) -- (c01) -- (c02) -- (cminus) -- (ooo);
\draw[thick, color=black] (ooo) to (c01);
\draw[thick, color=black] (ooo) to (c02);
\draw[thick, color=black] (ooo) to (c21);
\draw[thick, color=black] (ooo) to (cplus);
\draw[thick, color=black] (ooo) to (cminus);
\draw[thick, color=black] (cplus) to (c01);
\draw[thick, color=black] (cplus) to (c21);
\draw[thick, color=black] (c21) to (cminus);
\draw[thick, color=black] (c01) to (c02);
\path[fill, opacity=.9, color=gray!20] (ooo) -- (cplus) -- (c11) -- (c12) -- (cminus) -- (ooo);
\draw[thick, color=black] (ooo) to (c11);
\draw[thick, color=black] (ooo) to (c12);
\draw[thick, color=black] (cplus) to (c11);
\draw[thick, color=black] (c11) to (c12);
\draw[thick, color=black] (c02) to (cminus);
\draw[thick, color=black] (cminus) to (c12);
\end{tikzpicture}
}

\def\accrhothree{
\begin{tikzpicture}[scale=0.4]
\sffamily
\coordinate(ooo) at (0,0);
\coordinate(cplus) at (0,2);
\coordinate(c01) at (-1.4,1);
\coordinate(c02) at (-1.4,-1);
\coordinate(c11) at (1,.7);
\coordinate(c12) at (1,-.9);
\coordinate(c21) at (.8,1.8);
\coordinate(c22) at (.8,-.1);
\coordinate(cminus) at (0,-2);
\path[fill, color=gray!30] (ooo) -- (cplus) -- (c21) -- (c22) -- (cminus) -- (ooo);
\path[fill, color=gray!10] (ooo) -- (cplus) -- (c01) -- (c02) -- (cminus) -- (ooo);
\draw[thick, color=black] (ooo) to (c01);
\draw[thick, color=black] (ooo) to (c02);
\draw[thick, color=black] (ooo) to (c21);
\draw[thick, color=black] (ooo) to (c22);
\draw[thick, color=black] (ooo) to (cplus);
\draw[thick, color=black] (ooo) to (cminus);
\draw[thick, color=black] (cplus) to (c01);
\draw[thick, color=black] (cplus) to (c21);
\draw[thick, color=black] (c21) to (c22);
\draw[thick, color=black] (c22) to (cminus);
\draw[thick, color=black] (c01) to (c02);
\path[fill, opacity=.9, color=gray!20] (ooo) -- (cplus) -- (c11) -- (c12) -- (cminus) -- (ooo);
\draw[thick, color=black] (ooo) to (c11);
\draw[thick, color=black] (ooo) to (c12);
\draw[thick, color=black] (cplus) to (c11);
\draw[thick, color=black] (c11) to (c12);
\draw[thick, color=black] (c02) to (cminus);
\draw[thick, color=black] (cminus) to (c12);
\end{tikzpicture}
}

By the \emph{log canonicity} of a log terminal
projective surface $X$, we mean the number 
$\varepsilon_X := a_E+1$, where $a_E$ is the
minimal possible discrepancy appearing among the
exceptional divisors $E \subseteq \tilde X$ of its
minimal resolution of singularities.
Note that $1/\iota_X$ is bounded by the
log canonicity.
Alexeev's results~\cite{Al} show in particular
that bounding the log canonicity
gives finiteness for log del Pezzo surfaces.

\begin{corollary}
\label{cor:maxlc}
Consider a full intrinsic quadric surface
$X = X(P_\eta)$ with $P_\eta$ as in
Theorem~\ref{thm:fiqs2matrix-rho1},~\ref{thm:fiqs2matrix-rho2}
or~\ref{thm:fiqs2matrix-rho3}.
Then the log canonicity $\varepsilon_X$
of $X$ is given by
\[
\renewcommand{\arraycolsep}{3pt}
\begin{array}{lclcl}
\rho = 1:
&&
\varepsilon_X = \frac{1}{\iota^-}, 
\
\eta \in S_{11}(1,\iota) \cup S_{21}(1,\iota),
&&
\varepsilon_X = \frac{2}{\iota^-}, 
\
\eta \in S_{12}(1,\iota) \cup S_{22}(1,\iota),
\\[5pt]
\rho = 2:
&&
\varepsilon_X = \frac{3}{\iota^-}, 
\
\eta \in S_{11}(2,\iota) \cup S_{21}(2,\iota),
&&
\varepsilon_X = \frac{1}{\iota^-}, 
\
\eta \in S_{12}(2,\iota) \cup S_{22}(2,\iota),
\\[5pt]
\rho = 3:
&&
\varepsilon_X = \frac{2}{\iota^-}, 
\
\eta \in S_{11}(3,\iota) \cup S_{21}(3,\iota),
&&
\varepsilon_X = \frac{1}{\iota^-}, 
\
\eta \in S_{12}(3,\iota) \cup S_{22}(3,\iota).
\end{array}
\]
In particular, we obtain the following upper and lower
bounds for the log canonicity $\varepsilon_X$ of $X$:
\[
\rho = 1: \ \frac{1}{\iota} \le \varepsilon_X \le \frac{2}{\sqrt{\iota}},
\qquad  
\rho = 2: \ \frac{1}{\iota} \le \varepsilon_X \le \frac{3}{\sqrt{\iota}},
\qquad  
\rho = 3: \ \frac{1}{\iota} \le \varepsilon_X \le \frac{2}{\sqrt{\iota}}.
\]
\end{corollary}

\begin{proof}
For $X = X(P)$ with $P$ as in 
Construction~\ref{constr:matrix2fiqs-rho1},~\ref{constr:matrix2fiqs-rho2}
or~\ref{constr:matrix2fiqs-rho3}.
We use the anticanonical complex $\mathcal{A}_X$
introduced in~\cite{BeHaHuNi} to determine the
minimal discrepancies.
According to~\cite{HaHaSp}*{Thm.~9.17~(i) and~(ii)},
the maximal cells of
$\mathcal{A}_X$ are given in terms of the
columns $v_i$ of the matrix $P$ as 

{\small
\[
\begin{array}{cc}    
\begin{array}{ll}    
\rho=1: & \tilde v^+ = (a+1)e_3, \ \tilde v^- = (b+1)e_3,
\\[3pt]            
& \conv(0,\tilde v^+,v_1), \ \conv(0,v_1,v_2), \ \conv(0,v_2, \tilde v^-),
\\[3pt]
& \conv(0,\tilde v^+,v_3), \ \conv(0,v_3, \tilde v^-),
\\[3pt]
& \conv(0,\tilde v^+,v_3), \ \conv(0,v_3, \tilde v^-),  
\end{array}
&
\begin{array}{c}
\accrhoone
\end{array}
\\    
\\
\begin{array}{ll}    
\rho=2: & \tilde v^+ = \frac{2a+1}{3} e_3, \ \tilde v^- = \frac{2b+2c+1}{3} e_3,
\\[3pt]            
& \conv(0,\tilde v^+,v_1), \ \conv(0,v_1,v_2), \ \conv(0,v_2, \tilde v^-),
\\[3pt]
& \conv(0,\tilde v^+,v_3), \ \conv(0,v_3,v_4),\ \conv(0,v_4, \tilde v^-),
\\[3pt]
& \conv(0,\tilde v^+,v_5), \ \conv(0,v_5, \tilde v^-),  
\end{array}
&
\begin{array}{c}
\accrhotwo
\end{array}
\\            
\\
\begin{array}{ll}    
\rho=3: & \tilde v^+ = \frac{a}{2} e_3, \ \tilde v^- = \frac{b+c+d}{2}e_3,
\\[3pt]            
& \conv(0,\tilde v^+,v_1), \ \conv(0,v_1,v_2), \ \conv(0,v_2, \tilde v^-),
\\[3pt]
& \conv(0,\tilde v^+,v_3), \ \conv(0,v_3,v_4),\ \conv(0,v_4, \tilde v^-),
\\[3pt]
& \conv(0,\tilde v^+,v_5), \ \conv(0,v_5,v_6), \ \conv(0,v_6, \tilde v^-). 
\end{array}
&
\begin{array}{c}
\accrhothree
\end{array}
\end{array}
\]
}

\noindent
Now~\cite{HaHaSp}*{Thm.~9.17~(iii)} tells us that 
the discrepancies of the exceptional divisors $E^+$, $E^-$
given by the rays through $e_3$, $-e_3$ are given
for $\rho= 1,2,3$ by
\[
\frac{1}{a+1}-1, \  \frac{1}{-b-1}-1,
\qquad
\frac{3}{2a+1}-1, \  -\frac{3}{2b+2c+1}-1,
\qquad
\frac{2}{a}-1, \  -\frac{2}{b+c+d}-1.
\]
Moreover, these are obviously the minimal
discrepancies of the canonical resolution.
Inserting the values of $a,b,c,d$ from
Theorems~\ref{thm:fiqs2matrix-rho1},~\ref{thm:fiqs2matrix-rho2}
and~\ref{thm:fiqs2matrix-rho3},
we arrive at the assertion.
\end{proof}

The \emph{Picard index $\mathfrak{p}_X$} of a normal
variety $X$ is the index $[\Cl(X):\Pic(X)]$ of its
Picard group in its divisor class group.
Note that the Gorenstein index always divides the
Picard index.
Bounding the Picard yields finiteness for
del Pezzo surfaces of Picard number one
with torus action~\cite{Sp}; see also~\cite{HaHeSu}
for a higher dimensional analogue in the special
case of divisor class group $\ZZ$.

\goodbreak

\begin{corollary}
\label{cor:picind}
Consider a full intrinsic quadric surface
$X = X(P_\eta)$ with $P_\eta$ as in
Theorem~\ref{thm:fiqs2matrix-rho1},~\ref{thm:fiqs2matrix-rho2}
or~\ref{thm:fiqs2matrix-rho3}.
Then, according to the Picard number $\rho=\rho(X)$,
the Picard index $\mathfrak{p} = \mathfrak{p}_X$
of $X$ is given by
\[
\renewcommand{\arraycolsep}{5pt}
\begin{array}{llcll}
\rho = 1:
\\[5pt]
\mathfrak{p} = \frac{8 \iota^+\iota^-(\iota^+ + \iota^-)}{\gcd(2\iota^+, \, \iota^++\iota^-)},
&
\scriptstyle \eta \ \in \  S_{11}(1,\iota),   
&&
\mathfrak{p} = \frac{4 \iota^+\iota^-(2\iota^+ + \iota^-)}{\gcd(4\iota^+, \, 2\iota^++\iota^-)},
&
\scriptstyle \eta \ \in \  S_{12}(1,\iota),    
\\[8pt]
\mathfrak{p} = \frac{4 \iota^+\iota^-(\iota^+ + 2\iota^-)}{\gcd(2\iota^+, \, \iota^++2\iota^-)},
&
\scriptstyle \eta \ \in \  S_{21}(1,\iota),   
&&
\mathfrak{p} = \frac{2\iota^+\iota^-(\iota^+ + \iota^-)}{\gcd(2\iota^+, \, \iota^++\iota^-)},
&
\scriptstyle \eta \ \in \  S_{22}(1,\iota),    
\end{array}
%%%%%%% 
%%%%%%%
\]

\[
%%%%%%%  
%%%%%%%
\renewcommand{\arraycolsep}{4pt}
\begin{array}{llcll}
\rho = 2:
\\[5pt]
\mathfrak{p} = -\frac{c\iota^+\iota^-(\iota^+ + \iota^- + 2c)}{\gcd(2\iota^+, \,  \iota^+ + \iota^-, \,  2c)},
&
\scriptstyle \eta \ \in \  S_{11}(2,\iota),   
&&
\mathfrak{p} = -\frac{3c\iota^+\iota^-(\iota^+ + 3\iota^-+2c)}{\gcd(2\iota^+, \,  \iota^+ + 3\iota^-, \,  2c)},
&
\scriptstyle \eta \ \in \  S_{12}(2,\iota),    
\\[8pt]
\mathfrak{p} = -\frac{3c\iota^+\iota^-(3\iota^+ + \iota^- + 2c)}{\gcd(6\iota^+, \,  3\iota^+ + \iota^-, \,  2c)},
&
\scriptstyle \eta \ \in \  S_{21}(2,\iota),   
&&
\mathfrak{p} = -\frac{9c\iota^+\iota^-(3\iota^+ + 3\iota^-+2c)}{\gcd(6\iota^+, \,  3\iota^+ + 3\iota^-, \,  2c)},
&
\scriptstyle \eta \ \in \  S_{22}(2,\iota),    
\end{array}
%%%%%%% 
%%%%%%%
\]

\[
%%%%%%%  
%%%%%%%
\renewcommand{\arraycolsep}{4pt}
\begin{array}{llcll}
\rho = 3:
\\[5pt]
\mathfrak{p} = \frac{cd\iota^+\iota^-(\iota^+ + \iota^- + c +d)}{\gcd(\iota^+, \, \iota^-, \, c, \, d)},
&
\scriptstyle \eta \ \in \  S_{11}(3,\iota),   
&&
\mathfrak{p} = \frac{2cd\iota^+\iota^-(\iota^+ + 2\iota^- + c +d)}{\gcd(\iota^+, \, 2\iota^-, \, c, \, d)},
&
\scriptstyle \eta \ \in \  S_{12}(3,\iota),   
\\[8pt]
\mathfrak{p} = \frac{2cd\iota^+\iota^-(2\iota^+ + \iota^- + c +d)}{\gcd(2\iota^+, \, \iota^-, \, c, \, d)},
&
\scriptstyle \eta \ \in \  S_{21}(3,\iota),   
&&
\mathfrak{p} = \frac{4cd\iota^+\iota^-(2\iota^+ + 2\iota^- + c +d)}{\gcd(2\iota^+, \, 2\iota^-, \, c, \, d)},
&
\scriptstyle \eta \ \in \  S_{22}(3,\iota),
\end{array}
\]

\noindent
where $\iota$ is the Gorenstein index of $X$ and $\iota^\pm$ the local
Gorenstein index of $x^\pm \in X$ and $-c$ the local class group order
of $x_1 \in X$.
In particular, we obtain the following upper and lower bounds:
\[
\begin{array}{ll}
\rho=1: &  \iota \le \mathfrak{p} \le 8 \iota^2,
\\[5pt]
\rho=2: &  \iota \le \mathfrak{p} \le \frac{27}{2} \iota^3(3\iota-1),
\\[5pt]
\rho=3: &  \iota \le \mathfrak{p} \le 2 \iota^2(4\iota-1)^2(2\iota-1). 
\end{array}
\]
\end{corollary}

\begin{proof}
First assume $X = X(P)$ with $P$ from
Construction~\ref{constr:matrix2fiqs-rho1},~\ref{constr:matrix2fiqs-rho2}
or~\ref{constr:matrix2fiqs-rho3}.
Then Springer's formula~\cite{Sp}*{Thm.~1.1} gives us the Picard indices
\[
\textstyle
-\frac{8(a+1)(b+1)(a-b)}{\gcd(2a+2,a-b)},
\qquad
\frac{c(1+2a)(1+2b+2c)(a-b)}{\gcd(1+2a,a-b,c)},
\qquad
-\frac{acd(b+c+d)(a-b))}{\gcd(a,b,c,d)},
\]
according to the possible Picard numbers $\rho=1,2,3$. The assertion is
obtained by inserting the values of $a,b,c,d$ from
Theorems~\ref{thm:fiqs2matrix-rho1},~\ref{thm:fiqs2matrix-rho2}
and~\ref{thm:fiqs2matrix-rho3}.
\end{proof}

\goodbreak

A \emph{K\"ahler-Einstein metric} on a rational
projective del Pezzo surface is a K\"ahler
orbifold metric $g$ such that the associated
K\"ahler form~$\omega_g$ equals
its Ricci form $\mathrm{Ric}(\omega_g)$.
The smooth del Pezzo surfaces with
af K\"ahler-Einstein metric are
$\PP_2$, its blowing
up in $k = 3, \ldots, 8$ points in
general position and $\PP_1 \times \PP_1$;
see~\cites{TiYa,Ti1}.
The case of quasismooth del Pezzo
surfaces coming anticanonically embedded
into a three-dimensional weighted projective
space is understood as well;
see~\cites{KoJo,Arau,ChPaSh1,ChPaSh2}.
% Dropping the assumption of being anticanonically
% embedded or passing to complete intersections
% instead of hypersurfaces, the problem is
% still open; see for instance~\cites{KiPa,KiWo,LiuPe}.
We settle the case of full intrinsic quadric surfaces.

\goodbreak

\begin{corollary}
\label{cor:ke-metrics}
Let $X$ be a complex full intrinsic quadric surface
admitting a K\"ahler-Einstein metric.
Then $X \cong X(P)$ for precisely one $P$
from the following:

{\small
\[
\begin{array}{lcl}
\begin{array}{l}
\rho = 1, \ 2 \nmid \iota:
\\
\  
\end{array}
&\hspace*{1.2cm}&
\begin{array}{l}
\rho = 3,
\quad
2 \nmid \iota,
\quad
-2\iota \, \le \, 2c +d,
\\
c \, \le \, d \, \le \, -1,
\
c+d \, \le \, -\iota-1:
\end{array}
\\[10pt]
\ %P =   
\left[
\begin{array}{cccc}
-1 & -1 & 2 & 0 
\\
-1 & -1 & 0 & 2
\\
\iota-1 & -\iota-1 & 1 & 1 
\end{array}
\right],
&\hspace*{1.2cm}&
\ %P =   
\left[
\begin{array}{cccccc}
-1 & -1 & 1 & 1 & 0 & 0
\\
-1 & -1 & 0 & 0 & 1 & 1
\\
\iota & -\iota -c-d & 0 & c & 0 & d 
\end{array}
\right],
\\[20pt]
%
%
%\end{array}
%\]
%\[
%\begin{array}{lcl}
%
%
\begin{array}{l}
\rho = 1, \ 4 \mid \iota:
\\
\
\end{array}
&\hspace*{1.2cm}&
\begin{array}{l}
\rho=3,
\quad
-4\iota \, \le \, 2c +d, \ c \le d \, \le \, -1,
\\
c+d \, \le \, -2\iota -1:
\end{array}
\\[10pt]
\ %P =   
\left[
\begin{array}{cccc}
-1 & -1 & 2 & 0 
\\
-1 & -1 & 0 & 2
\\
\frac{\iota}{2}-1 & -\frac{\iota}{2}-1 & 1 & 1 
\end{array}
\right],
&\hspace*{1.2cm}&
\ %P =   
\left[
\begin{array}{cccccc}
-1 & -1 & 1 & 1 & 0 & 0
\\
-1 & -1 & 0 & 0 & 1 & 1
\\
2\iota & -2\iota -c -d & 0 & c & 0 & d 
\end{array}
\right],
\end{array}    
\]
}

\medskip
\noindent
where $\rho$ denotes the Picard number
and $\iota$ the Gorenstein index of $X(P)$.
Conversely, each $X(P)$ with $P$ from the
above list admits a K\"ahler-Einstein metric.
\end{corollary}

\begin{proof}
We may assume $X = X(P)$ with $P$ as in
Construction~\ref{constr:matrix2fiqs-rho1},~\ref{constr:matrix2fiqs-rho2}
or~\ref{constr:matrix2fiqs-rho3}.
This allows us to use the combinatorial $K$-stability
criterion for K\"ahler-Einstein metrics
for $X(P)$ provided by~\cite{HaHaSu};
see also~\cite{IlSu}.
The first step is to pick from the
toric degenerations $\mathcal{X}_\kappa \to \CC$
of $X(P)$, where $\kappa = 0,1,2$,
provided by~\cite{HaHaSu}*{Constr.~4.1}
the special ones, that means, those with
a normal central fiber $\mathcal{X}_{\kappa,0}$;
use~\cite{HaHaSu}*{Prop.~5.3}.
Then, with the aid of~\cite{HaHaSu}*{Prop.~5.6},
one computes the Fano polygons
$\mathcal{A}_\kappa$ of $\mathcal{X}_{\kappa,0}$, 
their duals $\mathcal{B}_\kappa$,
also called moment polytopes,
and the barycenters $b_\kappa$
of the $\mathcal{B}_\kappa$.
Finally,~\cite{HaHaSu}*{Thm.~6.2}
tells us that $X(P)$ admits a K\"ahler-Einstein metric
if and only if $b_{\kappa,1}=0$ and $b_{\kappa,2}>0$
hold for all special $\kappa$.

For Picard number $\rho=1$, consider $X = X(P)$
with $P$ as in Construction~\ref{constr:matrix2fiqs-rho1}.
In this setting, we obtain we obtain special
toric degenerations for $\kappa=1,2$ and we compute
for both cases
\[
\begin{array}{lcl}
\mathcal{A}_\kappa
& = &
\conv\left( (1,-2), \, (1+2a,2), \, (1+2b,b) \right),
\\[5pt]
\mathcal{B}_\kappa
& = &
\conv\left( (0,\frac{1}{2}), \, (-\frac{1}{1+b}, \frac{b}{2+2b}), \, (-\frac{1}{1+a}), \frac{a}{2+2a}) \right),
\\[5pt]
b_\kappa
& = & 
\left(-\frac{2+a+b}{3(1+a)(1+b)}, \, \frac{ab-1}{6(1+a)(1+b)} \right).
\end{array}
\]
Thus, $b_{\kappa,1}$ vanishes if and only if $b = -2-a$.
In this case, we have $b_{\kappa,2} = 1/6 > 0$.
Comparing with Theorem~\ref{thm:fiqs2matrix-rho1},
we arrive at the shapes given by $S_{11}(1,\iota)$ and
$S_{22}(1,\iota)$ with $\iota^+ = \iota^-$, where
$S_{12}(1,\iota)$, $S_{21}(1,\iota)$ are ruled out
by $\iota^+$, $\iota^-$ being odd.

For Picard number $\rho=2$, take $P$ as
in Construction~\ref{constr:matrix2fiqs-rho2}.
Then  $\kappa = 2$ yields a special degeneration
and we end up with barycenter $b_2 = (0,0)$, as
soon as $b_{2,1}=0$. Thus, none of the $X(P)$
admits a K\"ahler-Einstein metric.
For completeness we list the intermediate steps:
\[
\begin{array}{lcl}
\mathcal{A}_2
& = &
\conv\left( (1,-2), \, (a,1), \, (b+c,1) \right),
\\[5pt]
\mathcal{B}_2
& = &
\conv\left( (-\frac{3}{2a+1}, \frac{a-1}{2a+1}), \, (0,-1), \, (-\frac{3}{2b+2c+1}), \frac{b+c-1}{2b+2c+1}) \right),
\\[5pt]
b_2
& = & 
\left(-2\frac{a+b+c+1}{(2a+1)(2b+2c+1)}, \, -\frac{a+b+c+1}{(2a+1)(2b+2c+1)} \right).
\end{array}
\]

We turn to Picard number $\rho=3$. Let $X(P)$ arise from
Construction~\ref{constr:matrix2fiqs-rho3}.
Then we have special toric degenerations for
$\kappa=1,2,3$. The computation results are 

\[
\begin{array}{lcl}
\mathcal{A}_0
& = &
\conv\left( (0,1), \, (c+d,1), \, (b,-1), \, (a,-1) \right),
\\[7pt]
\mathcal{B}_0
& = &
\conv\left( (0,-1), \, \left( -\frac{2}{b+c+d}, \frac{c+d-b}{b+c+d} \right), \, (0,1), \, \left( -\frac{2}{a}, -1 \right) \right),
\\[7pt]
b_0
& = & 
\left( -\frac{2(a+b+c+d)}{3a(b+c+d)}, \, \frac{(b+2c+2d-a)b+(c+a+d)(c+d))}{3(a-b-c-d)(b+c+d)} \right),
\\[14pt]
\mathcal{A}_1
& = &
\conv\left( (a,1), \, (b+d,1), \, (c,-1), \, (0,-1) \right),
\\[7pt]
\mathcal{B}_1
& = &
\conv\left( (0,-1), \, \left( -\frac{2}{b+c+d}, \frac{b+d-c}{b+c+d} \right), \, (0,1), \, \left( -\frac{2}{a}, 1 \right)  \right),
\\[7pt]
b_1
& = & 
\left( -\frac{2(a+b+c+d)}{3a(b+c+d)}, \, -\frac{(a-b-2c-2d)b-(a+c+2d)c+(a-d)d) }{3(a-b-c-d)(b+c+d)} \right),
\\[14pt]
\mathcal{A}_2
& = &
\conv\left( (a,1), \, (b+c,1), \, (d,-1), \, (0,-1) \right),
\\[7pt]
\mathcal{B}_2
& = &
\conv\left( (0,-1), \, \left( -\frac{2}{b+c+d}, \frac{b+c-d}{b+c+d} \right), \, (0,1), \, \left( -\frac{2}{a}, 1 \right) \right),
\\[7pt]
b_2
& = & 
\left( -\frac{2(a+b+c+d)}{3a(b+c+d)}, \, -\frac{(a-b-2c-2d)b - (a+c-2d)c - (a+d)d }{ 3(a-b-c-d)(b+c+d) } \right).
\end{array}
\]

\medskip
\noindent
We conclude that $X(P)$ admits a K\"ahler-Einstein metric if and only if we have
$d=-a-b-c$, reflecting $b_{\kappa,1}=0$, and
\[
b > 0,   \quad
b+d < 0, \quad
a+d > 0, \quad
\]
reflecting $b_{\kappa,2}>0$. Substituting $a$, $b$ with the
corresponding entries from Theorem~\ref{thm:fiqs2matrix-rho3},
we arrive at
\[
\iota^+ \ = \ \iota^-,
\qquad
\iota^+ \ = \ 2\iota^-,
\qquad
2\iota^+ \ = \ \iota^-,
\qquad
\iota^+ \ = \ \iota^-,  
\]
according to the shapes defined by $S_{11}(3,\iota)$, $S_{12}(3,\iota)$, $S_{21}(3,\iota)$ and $S_{22}(3,\iota)$.
Note that $S_{12}(3,\iota)$, $S_{21}(3,\iota)$ are ruled out by $\iota^+$, $\iota^-$ being odd,
respectively.
\end{proof}

% \begin{corollary}
% Every full intrinsic quadric surface $X$ of Picard number
% one admits a K\"ahler-Ricci soliton.
% \end{corollary}

\

\end{document}